\documentclass{article}

\usepackage[preprint]{neurips_2026}

\usepackage[utf8]{inputenc} % allow utf-8 input
\usepackage[T1]{fontenc}    % use 8-bit T1 fonts
\usepackage{hyperref}       % hyperlinks
\usepackage{url}            % simple URL typesetting
\usepackage{booktabs}       % professional-quality tables
\usepackage{amsfonts}       % blackboard math symbols
\usepackage{nicefrac}       % compact symbols for 1/2, etc.
\usepackage{microtype}      % microtypography
\usepackage{xcolor}         % colors
\usepackage{amsmath, amssymb, amsthm}
\usepackage{tikz-cd}
\usepackage{enumitem}
\usepackage{graphicx}

\newtheorem{definition}{Definition}
\newtheorem{proposition}{Proposition}
\newtheorem{theorem}{Theorem}
\newtheorem{lemma}{Lemma}
\newtheorem{remark}{Remark}
\newtheorem{assumption}{Assumption}
\newtheorem{corollary}{Corollary}

\AtBeginEnvironment{algorithm}{\SetInd{.5em}{1em}}
\usepackage[boxruled]{algorithm2e}
\usepackage{algorithmic}

\title{Monotone Inclusion Approach to Weakly Monotone Discrete-Time Finite-Horizon Mean-Field Games\thanks{Research of first and second authors was supported in part by the AFOSR Grant FA9550-24-1-0152}}

\author{%
U\u{g}ur Ayd{\i}n \\ Department of Electrical and Computer Engineering \\ University of Illinois Urbana-Champaign \\ Urbana, IL 61801, USA \\ \texttt{uaydin2@illinois.edu} \\ \And Tamer Ba\c{s}ar \\ Department of Electrical and Computer Engineering \\ University of Illinois Urbana-Champaign \\ Urbana, IL 61801, USA \\ \texttt{basar1@illinois.edu} \\ \And Naci Saldi \\ Department of Mathematics \\ Bilkent University \\ \c{C}ankaya, Ankara 06800, Turkey \\ \texttt{naci.saldi@bilkent.edu.tr} \\ }

\begin{document}

\maketitle

\begin{abstract}
We revisit the problem of computing mean-field equilibria (MFEs) in discrete-time, monotone, finite-horizon mean-field games (MFGs). We show that, when the transition kernel is independent of the state-measure term and the reward function satisfies the usual weak monotonicity condition and is Lipschitz continuous, anchored proximal   gradient descent methods can be used to compute a monotone MFE. We also establish last-iterate convergence results for these methods.
Our approach relies on formulating the computation problem as an optimization problem over the space of occupation measures. Using this formulation, we show that the problem is equivalent to a class of constrained Lipschitz monotone inclusion problems. We then apply iterative methods for this monotone inclusion formulation to derive a tractable algorithm.
The resulting algorithm achieves a convergence rate of \(O(1/\sqrt{T})\) after \(T\) iterations, without requiring any regularization. This rate holds even in the absence of a uniqueness assumption for the corresponding MFE.
\end{abstract}

\section{Introduction}
Mean-field games (MFGs) provide a powerful framework for approximating solutions to multi-agent problems with identical agents. The theory of MFGs was introduced independently by Lasry and Lions in \cite{LaLi07}, who coined the term ``mean-field games,'' and by Huang, Malham\'{e}, and Caines in \cite{HuMaCa06}, who formulated the problem from the viewpoint of stochastic dynamic games. These foundational works study continuous-time noncooperative differential games with a large but finite number of agents, where the influence of each individual agent on the population becomes asymptotically negligible as the number of agents grows. In this paper, we consider the discrete-time counterpart of this framework, introduced in \cite{SaBaRaSIAM}.

MFGs reduce multi-agent problems with identical agents to an effective single-agent problem and provide approximate solutions to general Markov games, for which the existence of a solution is not always guaranteed \cite{SaBaRaSIAM}. Motivated by the growing interest in multi-agent systems, computational methods for MFGs have also attracted significant attention \cite{guo2019learning}. However, the dynamics of MFGs generally do not satisfy any natural contraction property. In fact, computing equilibria in most MFGs is PPAD-complete \cite{yardim2024mean}, suggesting that the problem is computationally intractable in general.

In this work, we revisit the learning problem for monotone finite-horizon mean-field games. We assume only that the transition probability is independent of state measures and the reward function is Lipschitz continuous on the space of state measures and satisfies the so-called weak monotonicity condition. Under these assumptions, we reformulate the problem as a minimization problem over the space of occupation measures on \(S \times A\), where \(S\) denotes the state space and \(A\) the action space. By expressing the reward function as an operator of the occupation measure, we show that the problem of finding a mean-field equilibrium (MFE) can be reformulated as a monotone inclusion problem of the form
\(
\mathbf{0} \in F(\mathbf{z}) + A(\mathbf{z}),
\)
where \(F\) is a single-valued, monotone, and Lipschitz operator, and \(A\) is a maximal monotone operator defined over the space of feasible occupation measures.

The inclusion problem $\pmb 0 \in F(\pmb z)+A(\pmb z)$ in general cannot be solved via vanilla descent methods \cite{daskalakis2018limit}. To design iterative an algorithm for this inclusion problem, in the optimization literature, it is well known that one needs further structure, such as optimisim, as in extragradient methods, or accelaration methods, as in anchoring \cite{golowich2020last} \cite{ryu2019ode}. For simplicity, we will merely work under anchoring in this work. This will allow us to present a completely tractable algorithm by using the recent result \cite[Theorem 2]{cai2026last} with convergence guarantees without relying on contraction or regularizers.

Two primary structural conditions are commonly used for computational purposes: contractivity and monotonicity. Unfortunately, monotonicity-based approaches often require additional regularization structure and typically apply only in fairly restrictive settings. In the MFG literature, monotonicity conditions are most often used to establish uniqueness of equilibrium \cite{saldi2020approximate}\cite{perolat2021scaling}. Beyond uniqueness, such conditions have also found applications in the mean-field reinforcement learning (MFRL) literature \cite{zhang2023learning}\cite{isobe2024last}. In this paper, we assume only that the reward function is weakly monotone and Lipschitz continuous on the space of state measures.

Previous work on last-iterate convergence for such methods often relies on proximal techniques, for example, by tracking the proximal operator as the base point changes \cite{isobe2024last}. In contrast, we show that the anchoring mechanism allows us to avoid introducing additional regularization terms, which can otherwise lead to deviations from the true MFE. We also establish a convergence rate for the objective function. To the best of our knowledge, this is the first algorithm in the literature for the non-regularized, non-contractive setting that admits an explicit convergence rate. We further emphasize that our framework allows for the presence of multiple MFEs, and our convergence guarantee remains valid even in this non-unique setting.

\subsection{Related Work}

The main assumption imposed in this work is the so-called monotonicity condition. Monotonicity conditions are widely used in the study of non-stationary MFEs, as they are among the few structural assumptions known to guarantee uniqueness of MFE (\cite{saldi2020approximate,perolat2021scaling}). Nevertheless, the ``weak'' monotonicity condition by itself is generally insufficient to ensure uniqueness. In the computational literature, monotonicity has mainly been exploited in algorithms based on policy updates (\cite{isobe2024last,zhang2023learning}). By contrast, we formulate the problem as an optimization problem over the space of measures and develop an algorithm directly on the space of occupation measures.

Representing the underlying MFE problem as an optimization problem over the corresponding space of occupation measures is a standard technique in the study of Markov decision processes (\cite{puterman2014markov}). Although the optimization formulation was not explicitly used there, occupation measures were first employed in the context of discrete-time MFGs to establish the existence of MFEs in \cite{SaBaRaSIAM}. Later, \cite{guo2022optimization} used the associated optimization problem to propose a framework for solving finite-horizon MFEs. However, the algorithms developed in that work are intractable and rely on unknown initialization constraints as well as higher-order differentiability assumptions on the system components. For ``linear'' MFGs, \cite{saldi2026linear} also developed optimization-based algorithms for computing MFEs using occupation measures. None of these works provides an algorithm with an explicit convergence rate.

More recently, the works \cite{anahtarci2025maximum,alkir2025inverse, guo2022optimization,anahtarci2025kernel} developed reinforcement learning algorithms based on the occupation measure approach. The key distinction between our method and these previous results is that those works formulate the optimization problem through the space of policies. As a consequence, the required subproblems are generally intractable. In contrast, we formulate the problem directly as an optimization problem over occupation measures and develop an iterative algorithm that operates entirely in the space of occupation measures, without invoking policies even at the level of subroutines. This subtle but important distinction allows us to reformulate the problem as a monotone inclusion problem over occupation measures.

By representing MFE as an optimization problem, we identify an operator \(F\) on the space of occupation measures. Since the problem must be restricted to feasible occupation measures, we are naturally led to the constrained monotone inclusion
\(
\mathbf{0} \in F(\mathbf{z}) + A(\mathbf{z}),
\)
where \(A\) is a set-valued operator that captures the constraints induced by the occupation measure formulation. The unconstrained case, corresponding to \(A \equiv 0\), has recently attracted considerable attention because of its close connection to saddle-point problems arising in the training of generative adversarial networks (GANs) (\cite{golowich2020last};\cite{ryu2019ode}). The constrained case \(A \not\equiv 0\), however, presents additional technical challenges (\cite{cai2023doubly}). Even in the unconstrained setting, standard proximal gradient methods may fail to converge (\cite{daskalakis2018limit}). To address this issue, methods based on \emph{optimism} and \emph{acceleration} have been incorporated into proximal gradient schemes (\cite{yoon2021accelerated}).

Finally, we briefly note that existing algorithms for weakly monotone MFGs are typically based on policy updates, which are often implemented through KL-divergence regularization or related techniques (\cite{zhang2023learning};\cite{isobe2024last};\cite{perolat2021scaling}). By contrast, the algorithm introduced in this work operates solely on the space of occupation measures and therefore avoids the use of auxiliary policy-update procedures such as proximal methods through KL divergence regularization.

\subsection{Contributions}\label{sect:cont}

Our main contributions are summarized as follows:
\begin{enumerate}
    \item We provide a maximal monotone set-inclusion characterization of MFEs under weak monotonicity of the reward function.
    \item When the transition probability kernel is independent of the population distribution, we show that the monotone set-inclusion problem corresponding to the MFE problem can be solved in polynomial time using the anchored proximal algorithm introduced by \cite{cai2026last}. It appears that this is the first fully tractable and provably convergent algorithm for computing exact finite-horizon mean-field equilibria in discrete-time mean-field games.
    \item Tractability of our algorithm relies on solving a convex quadratic program exactly at each iteration, which can be done in polynomial time. Using the interior-point primal--dual path-following algorithm of \cite{monteiro1989interior}, we provide an exact arithmetic complexity bound for solving this quadratic program.
    \item To the best of our knowledge, the algorithm introduced here is the first in the discrete-time MFG literature to achieve an explicit convergence rate in the non-regularized setting while allowing for multiple mean-field equilibria. In addition, under our assumptions, we construct an example of an MFG showing that the upper bound of our algorithm is tight, in the sense that it is matched by a corresponding lower bound for this algorithm.
\end{enumerate}

A comparison of our contributions we have listed above with (some) of the previous literature can be found in the table below.

\begin{center}
\renewcommand{\arraystretch}{1.4}
\resizebox{\linewidth}{!}{%
\begin{tabular}{llllll}
\hline
Work & Assumption & Type & Rate of Convergence & Uniqueness & Tractable\\
\hline
\cite{guo2019learning} & Contraction & Regularized & Exponential & Yes & Yes \\
\cite{anahtarci2020value} & Contraction & Nonregularized & Exponential & Yes & Yes\\
\cite{anahtarci2023q} &  Contraction & Regularized & Exponential & Yes & Yes\\
\cite{cui2021approximately} & Contraction & Regularized & Exponential & Yes & Yes \\
\cite{yardim2023policy} & Contraction & Regularized & Exponential & Yes & Yes\\
\cite{zhang2023learning} & Strict Mono. & Regularized & $O(\log^2 T/\sqrt T)$ & Yes & Yes \\
\cite{zhang2024stochastic} & Contraction & Regularized & Exponential & Yes & Yes \\
\cite{guo2024mf} & Differentiable & Nonregularized & N/A & No & No \\
\cite{dong2025last} & Weak Mono. & Regularized & $O(1/T)$ & No & Yes\\
\cite{isobe2024last} & Weak Mono. & Nonregularized& N/A & No & No \\
\bf{This work} & \bf{Weak Mono.} & \bf{Nonregularized} & $O(1/\sqrt T)$ & \bf{No} & \bf{Yes} \\
\hline
\end{tabular}
}
\end{center}

\begin{remark}
    The work \cite{dong2025last} establishes an \(O(1/T)\) convergence rate when one has oracle access to the reward function and the transition kernel. In addition, \cite{isobe2024last} shows that exponential convergence is achievable in the regularized setting.
\end{remark}

\section{Preliminaries}

\subsection{Discrete-Time Mean-Field Games}\label{sect:2.1}

In this subsection, following the framework of \cite{SaBaRaSIAM}, we introduce the setup for discrete-time mean-field games. A discrete-time finite-horizon mean-field game is specified by a tuple \((S,A,r,P,\mu_0,H)\), where
\begin{enumerate}
    \item \(S\) is a finite state space.
    \item \(A\) is a finite action space.
    \item \(r:S \times A \times \mathcal P(S) \to \mathbb R\) is the one-stage reward function, where \(\mathcal P(S)\) denotes the space of probability measures on \(S\).
    \item \(P:S \times A \times \mathcal P(S) \to \mathcal P(S)\) is the one-stage transition kernel.
    \item \(\mu_0 \in \mathcal P(S)\) is the initial state distribution.
    \item \(H\) is the horizon length.
\end{enumerate}

Each element \(\mu \in \mathcal P(S)\) will be referred to as a \emph{state measure}, representing the distribution of states across a continuum of identical agents. A \emph{state-measure flow} is a tuple of state measures \(\pmb{\mu} = (\mu_t)_{t=0}^{H-1}\). Throughout the paper, whenever an MFG is fixed, we implicitly assume that every state-measure flow starts from the initial state measure \(\mu_0\) specified in the definition of the game.

For a given finite horizon \(H \in \mathbb{N}\), the objective in an MFG is to find a policy flow \(\pmb{\pi}^* = (\pi_h^*)_{h=0}^{H-1}\) and a state-measure flow \(\pmb{\mu}^* = (\mu_h^*)_{h=0}^{H-1}\) such that the objective
\begin{equation}\label{eq:base}
J(\pmb{\pi}^*,\pmb{\mu}^*) := \mathbb{E}\left[\sum_{h=0}^{H-1} r(x_h,a_h,\mu_h^*)\right],
\end{equation}
where \(x_{h+1} \sim P(\cdot \mid x_h,a_h,\mu_h^*)\), \(a_h \sim \pi_h^*(\cdot \mid x_h)\), and \(x_0 \sim \mu_0\), satisfies
\begin{equation}\label{eq:first}
J(\pmb{\pi},\pmb{\mu}^*) \leq J(\pmb{\pi}^*,\pmb{\mu}^*)
\end{equation}
for every policy flow \(\pmb{\pi}\), and the state-measure flow \(\pmb{\mu}^*\) satisfies the \emph{consistency condition}
\begin{equation}\label{eq:last}
\mu_{h+1}^*(\cdot)
=
\sum_{s \in S}\sum_{a \in A}
P(\cdot \mid s,a,\mu_h^*)\,\pi_h^*(a \mid s)\,\mu_h^*(s),
\qquad h=0,\dots,H-2.
\end{equation}
When \(r\) and \(P\) are independent of the state-measure component in \(\mathcal{P}(S)\), the problem reduces to a standard Markov decision process. If the pair \((\pmb{\pi}^*,\pmb{\mu}^*)\) satisfies \eqref{eq:first} and \eqref{eq:last}, then we say that \((\pmb{\pi}^*,\pmb{\mu}^*)\) is a \emph{mean-field equilibrium}.

The main difficulty in computing an MFE arises from the dependence of \(r\) and \(P\) on the state-measure component \(\mathcal P(S)\), together with the consistency requirement in \eqref{eq:last}. Because of this additional dependence and the associated evolution condition, there is generally no natural contractive dynamical system available for computational purposes. We also emphasize that a discrete-time MFG is, in general, neither a standard game nor a standard control problem, but rather lies somewhere in between.

For our purposes, we impose the following assumption on the MFG.

\begin{assumption}[Weak monotonicity]\label{weakmon}
For every pair of policy flows \(\pmb{\pi}=(\pi_h)_{h=0}^{H-1}\) and \(\pmb{\tilde{\pi}}=(\tilde{\pi}_h)_{h=0}^{H-1}\), it holds that
\begin{equation}\label{mon}
\sum_{h=0}^{H-1} \sum_{s\in S}\sum_{a\in A}
\Bigl(r(s,a,\mu_h^{\pmb{\pi}})-r(s,a,\mu_h^{\pmb{\tilde{\pi}}})\Bigr)
\Bigl(\rho_h^{\pmb{\pi}}(s,a)-\rho_h^{\pmb{\tilde{\pi}}}(s,a)\Bigr)\le 0,
\end{equation}
where $\pmb \mu^{\pmb \pi}$ and $\pmb \mu^{\pmb{\tilde \pi}}$ are state-measure flows such that
\[
\rho_h^{\pmb{\pi}}(s,a)=\pi_h(a\mid s)\mu_h^{\pmb{\pi}}(s),
\qquad
\rho_h^{\pmb{\tilde{\pi}}}(s,a)=\tilde{\pi}_h(a\mid s)\mu_h^{\pmb{\tilde{\pi}}}(s).
\]
\end{assumption}

We say that the monotonicity condition is \emph{strict} if \eqref{mon} holds with equality only when \(\pmb{\pi}=\pmb{\tilde{\pi}}\). In general, strict monotonicity \cite[Proposition 1]{perolat2021scaling}, or alternatively additional separability conditions on the reward function \cite{saldi2020approximate}, is needed to ensure uniqueness of the MFE. Therefore, under Assumption \ref{weakmon}, the finite-horizon MFE need not be unique in general.

We further note that Assumption \ref{weakmon} guarantees the existence of an MFE; see \cite[Proposition 1]{perolat2021scaling}. Without monotonicity, existence typically requires additional regularity conditions on the system components \(r\) and \(P\), most notably uniform continuity; see \cite{SaBaRaSIAM}.

\begin{assumption}\label{ass:2}
    The transition kernel \(P\) is independent of the state-measure component; namely,
    \[
    P(\cdot \mid s,a,\mu)=P(\cdot \mid s,a)
    \qquad \text{for all } \mu \in \mathcal{P}(S).
    \]
\end{assumption}

Assumption \ref{ass:2}, together with Assumption \ref{weakmon}, is frequently imposed to guarantee uniqueness of the MFE \cite{saldi2020approximate}. It is also adopted in \cite{isobe2024last} to develop an algorithm in the weakly monotone setting. On the other hand, under strict monotonicity, the framework of \cite{zhang2023learning} does not require Assumption \ref{ass:2}. In our setting, the main role of Assumption \ref{ass:2} is that it renders the constraint set in the optimization formulation a convex polytope, a property that is essential for the development of a fully tractable algorithm. This assumption also occurs commonly in practice, see \cite{weintraub2008markov,subramanian2019reinforcement}.

\begin{assumption}[Lipschitz continuity]\label{lip}
    For all \((x,a) \in S \times A\) and all \(\mu,\tilde{\mu} \in \mathcal{P}(S)\), we have
    \[
    |r(x,a,\mu)-r(x,a,\tilde{\mu})|
    \leq L_r \|\mu-\tilde{\mu}\|_{\mathrm{TV}},
    \]
    where \(\|\cdot\|_{\mathrm{TV}}\) denotes the total variation norm on \(\mathcal{P}(S)\).
\end{assumption}

The Lipschitz continuity assumption imposed above is weaker than the assumptions typically used in contractive approaches for the nonlinear case, where \(r\) and \(P\) are also required to be Lipschitz continuous with respect to the state and action variables; see \cite{anahtarci2023q,yardim2023policy,guo2019learning,cui2021approximately}. In contrast, the work \cite{zhang2023learning} does not require any Lipschitz continuity assumption under strict monotonicity. As noted earlier, since we establish convergence rates for first-order methods, it is natural to impose additional regularity conditions on the system components, such as Lipschitz continuity.

\begin{remark}
    For simplicity, we present our main results in the time-homogeneous setting; that is, for all \(h=0,1,\dots,H-1\),
    \(
    P_h(\cdot \mid \cdot,\cdot) \equiv P(\cdot \mid \cdot,\cdot)\)
    and 
    \(r_h(\cdot,\cdot,\cdot) \equiv r(\cdot,\cdot,\cdot).\)
    However, our results continue to hold in the time-inhomogeneous setting as well.
\end{remark}

\subsection{Monotone Operators and Inclusions}

Our computational results rely on representing an MFE as an optimization problem involving monotone operators. The purpose of this subsection is to provide the basic background on monotone operators and monotone inclusions needed to make the presentation as self-contained as possible.

For our purposes, all monotone operators are finite-dimensional. Throughout the paper, unless stated otherwise, every finite-dimensional Euclidean space \(\mathbb{R}^d\) is equipped with the Euclidean norm \(\|\cdot\|_2\).

\begin{definition}[Monotone Operator]
Let $\mathcal H$ be a finite-dimensional Hilbert space with inner product
$\langle \cdot,\cdot\rangle$, and let $A:\mathcal H \rightrightarrows \mathcal H$
be a set-valued operator. The graph of $A$ is
\(
\operatorname{gra}(A)
:=
\{(\pmb x,\pmb u)\in \mathcal H\times \mathcal H : \pmb u\in A(\pmb x)\}.
\) 
The operator $A$ is called \emph{monotone} if
\[
\langle \pmb u-\pmb v,\,\pmb x- \pmb y\rangle \ge 0
\qquad
\forall (\pmb x,\pmb u),(\pmb y,\pmb v)\in \operatorname{gra}(A).
\]
\end{definition}

\begin{definition}[Maximal Monotone Operator]
The set-valued operator $A:\mathcal H \rightrightarrows \mathcal H$ is called \emph{maximal monotone} if it is monotone and
there is no monotone set-valued operator $B:\mathcal H \rightrightarrows \mathcal H$
such that
\(
\operatorname{gra}(A)\subsetneq \operatorname{gra}(B).
\)
\end{definition}

The following remark says that the normal cone of a compact convex set about a given point is a maximal monotone operator, which will be the case of interest for us.

\begin{remark}\label{rem:1}[Normal cone is maximal monotone]
Let $\mathcal K\subset \mathcal H$ be a nonempty compact convex set, and define its normal
cone operator by
\[
N_{\mathcal K}(\pmb x)
:=
\begin{cases}
\{\,\pmb g\in \mathcal H : \langle \pmb g,\pmb y-\pmb x\rangle \le 0 \ \forall \pmb y\in \mathcal K\,\}, & \pmb x\in \mathcal K,\\[1mm]
\varnothing, & \pmb x\notin \mathcal K.
\end{cases}
\]
Then, $N_{\mathcal K}$ is a maximal monotone operator.

Indeed, since $\mathcal K$ is nonempty, compact, and convex, it is in particular closed and convex.
Hence the indicator function
\[
I_{\mathcal K}(\pmb x)
:=
\begin{cases}
0, & \pmb x\in \mathcal K,\\
+\infty, & \pmb x\notin \mathcal K
\end{cases}
\]
is proper, lower semicontinuous, and convex \cite{rockafellar1970convex}. Moreover,
\[
\partial I_{\mathcal K}(\pmb x)=N_{\mathcal K}(\pmb x),
\]
where $\partial I_{\mathcal K}$ denotes the convex subdifferential \cite[Example 16.12]{bauschke2020correction}. Since the subdifferential
of any proper lower semicontinuous convex function is maximal monotone, it follows
that $N_{\mathcal K}$ is maximal monotone \cite[Example 20.41]{bauschke2020correction}.
\end{remark}

For a given maximal monotone set-valued operator $A : \mathbb R^d \rightrightarrows \mathbb R^d$ and a single-valued Lipschitz monotone operator $F : \mathbb R^d \to \mathbb R^d$, in this paper we are interested in finding $\pmb z \in \mathbb R^d$ such that
\(
\pmb 0 \in F(\pmb z) + A(\pmb z),
\)
which is called the (constrained) \emph{monotone inclusion problem}.

\section{Discrete-Time Mean-Field Equilibrium as a Monotone Inclusion Problem}\label{sect:3}

Let \(\mathrm{MFG}_{\mu_0}=(S,A,r,P,\mu_0,H)\) be a finite-horizon mean-field game. In this section, we represent an MFE of \(\mathrm{MFG}_{\mu_0}\) as a solution to a monotone inclusion problem. To do so, we formulate the MFE problem as an optimization problem over the space of occupation measures. In the next section, we show that this optimization problem can be solved by tractable algorithms, and we present a tractable algorithm together with convergence guarantees.

To represent an MFE as a monotone inclusion problem, we first formulate the MFE problem as an optimization problem. For this purpose, we introduce additional notation and definitions related to occupation measures.

\begin{definition}
Let \(\pmb{\pi}\) be a policy flow, and let \(\pmb{\mu}^{\pmb \pi}=(\mu_h^{\pmb \pi})_{h=0}^{H-1}\) be a state-measure flow associated with \(\pmb{\pi}\). For each \(h=0,\dots,H-1\), the joint probability measure on \(S \times A\) defined by
\[
\rho_h^{\pmb{\pi}}(s,a):=\pi_h(a\mid s)\,\mu_h^{\pmb{\pi}}(s)
\]
is called the \emph{occupation measure} corresponding to \(\pmb{\pi}\) at time \(h\). The collection \(\pmb{\rho}^{\pmb{\pi}}=(\rho_h^{\pmb{\pi}})_{h=0}^{H-1}\) is called the corresponding \emph{occupation flow}.
\end{definition}

For a given horizon length \(H\), we define the ambient space containing all occupation flows corresponding to admissible policies as the finite-horizon occupation-flow polytope
\[
\mathcal{K}:=
\left\{
\pmb{\rho}=(\rho_h)_{h=0}^{H-1} \in \prod_{h=0}^{H-1}\mathbb{R}_+^{S\times A}:
\begin{array}{l}
\displaystyle \sum_{a\in A}\rho_0(s,a)=\mu_0(s)\quad \forall s\in S,\\[2mm]
\displaystyle \sum_{a\in A}\rho_{h+1}(s',a)=
\sum_{s\in S}\sum_{a\in A}P(s'\mid s,a)\rho_h(s,a)\\
\hfill \forall s'\in S,\ \forall h=0,\dots,H-2
\end{array}
\right\}.
\]

In general, a policy flow need not determine a unique occupation flow. Conversely, a given occupation flow may correspond to more than one policy flow.

\begin{definition}
For a given occupation flow \(\pmb{\rho}\in\mathcal{K}\), the \emph{induced state flow} is defined by
\[
\mu_h^{\pmb{\rho}}(s):=\sum_{a\in A}\rho_h(s,a),
\qquad h=0,\dots,H-1.
\]
\end{definition}

It is straightforward to verify that if \((\pmb{\pi}^*,\pmb{\mu}^*)\) is a finite-horizon MFE, then there exists an occupation flow corresponding to \(\pmb{\pi}^*\) whose induced state flow is exactly \(\pmb{\mu}^*\).

\begin{lemma}\label{lem:1}
    Suppose that Assumption \ref{ass:2} holds. Then, \(\mathcal{K}\) is a nonempty, compact, convex polytope in the Euclidean space under consideration.
\end{lemma}

Assumption \ref{ass:2} is crucial for showing that \(\mathcal{K}\) is convex. For our purposes, the importance of this convexity property is that it implies that the normal cone \(N_{\mathcal K}(\pmb{x})\) is maximal monotone for every \(\pmb{x}\in\mathcal{K}\); see Remark \ref{rem:1}. Although \cite{guo2024mf} also formulates an optimization problem based on occupation measures, it does not consider the setting of Assumption \ref{ass:2}. In the absence of Assumption \ref{ass:2}, the corresponding optimization problem for MFE typically becomes nonlinear, and this nonlinearity is difficult to handle in full generality.

\begin{definition}
The operator \(F:\mathcal{K}\to\mathbb{R}^{H|S||A|}\), defined coordinatewise by
\[
[F(\pmb{\rho})]_{h,s,a}:=-\,r\bigl(s,a,\mu_h^{\pmb{\rho}}\bigr),
\]
is called the \emph{mean-field occupation operator}.
\end{definition}

It is immediate that \(F\) is a single-valued operator. For our purposes, we will also need \(F\) to be monotone and Lipschitz continuous on \(\mathcal{K}\).

\begin{lemma}\label{F-Lip}
    Suppose Assumptions \ref{weakmon}--\ref{lip} hold. Then, for all \(\pmb{\rho},\pmb{\tilde{\rho}} \in \mathcal K\),
    \[
    \|F(\pmb{\rho})-F(\pmb{\tilde{\rho}})\|_2
    \le |S||A|\,L_r\,\|\pmb{\rho}-\pmb{\tilde{\rho}}\|_2.
    \]
\end{lemma}

\begin{lemma}\label{F-mon}
    Suppose Assumptions \ref{weakmon}--\ref{lip} hold. Then, the operator \(F\) is monotone on \(\mathcal K\); that is,
    \[
    \langle F(\pmb{\rho})-F(\pmb{\tilde{\rho}}),\,\pmb{\rho}-\pmb{\tilde{\rho}}\rangle \ge 0
    \qquad \forall \pmb{\rho},\pmb{\tilde{\rho}}\in \mathcal K.
    \]
\end{lemma}

The main result of this section is the following.

\begin{theorem}\label{thrm:main}
    Suppose that Assumptions \ref{weakmon}--\ref{lip} hold. Let \((\pmb{\pi^*},\pmb{\mu^*})\) be a policy flow and a state-measure flow, respectively, and define
    \[
    \rho_h^*(s,a)=\pi_h^*(a\mid s)\mu_h^*(s).
    \]
    Then, \((\pmb{\pi^*},\pmb{\mu^*})\) is an MFE if, and only if,
    \begin{equation}\label{eq:vi-main}
    \langle F(\pmb{\rho^*}),\,\pmb{\rho}-\pmb{\rho^*}\rangle \ge 0
    \qquad \forall \pmb{\rho}\in \mathcal K.
    \end{equation}
\end{theorem}

\begin{proof}[Outline of the proof.]
For any policy flow \(\pmb{\pi}\) and state-measure flow \(\pmb{\mu}\), since the transition kernel is independent of the state-measure flow, we can write
\begin{equation}\label{J}
J(\pmb{\pi},\pmb{\mu})
=
\sum_{h=0}^{H-1}\sum_{s\in S}\sum_{a\in A}
r(s,a,\mu_h)\,\rho_h^{\pmb{\pi}}(s,a) =: L_{\pmb \mu}(\pmb {\rho^{\pi}}).
\end{equation}
For any fixed state-measure flow \(\pmb \mu\), we extend the functional $L_{\pmb \mu}$ to entire $\mathcal K$ as
\[
L_{\pmb \mu}(\pmb{\rho})
=
\sum_{h=0}^{H-1}\sum_{s\in S}\sum_{a\in A}
r\bigl(s,a,\mu_h\bigr)\, \rho_h(s,a).
\]
Using the functional \(L_{\pmb \mu}(\pmb{\rho})\), we show that
\((\pmb {\pi^*},\pmb {\mu^*})\) induces an MFE if, and only if,
\(
L_{\pmb {\mu^*}}(\pmb \rho)\le L_{\pmb {\mu^*}}(\pmb {\rho^*})\) for all $\pmb \rho \in \mathcal K$,
where \(\pmb{\rho^*}\) is the occupation flow corresponding to this MFE.
Moreover, using \eqref{J}, one can readily show that this criterion is
equivalent to \eqref{eq:vi-main}.
\end{proof}

It can be shown that \(\mathcal K\) is a nonempty, compact, and convex set. Equivalently, \eqref{eq:vi-main} can be written as the monotone inclusion \cite[Section 1.1]{sedlmayer2023fast}
\begin{equation}\label{eq:mi-main}
\mathbf{0} \in F(\pmb{\rho^\star}) + N_{\mathcal K}(\pmb{\rho^\star}),
\end{equation}
where \(N_{\mathcal K}\) denotes the normal-cone operator of \(\mathcal K\) (Remark \ref{rem:1}).

\begin{corollary}\label{cor:1}
With the notation of Theorem \ref{thrm:main}, \((\pmb{\pi^*}, \pmb{\mu^*})\) is a finite-horizon MFE if, and only if,
\[
\mathbf{0} \in F(\pmb{\rho^\star}) + N_{\mathcal K}(\pmb{\rho^\star}),
\]
where $\rho^*_h(s,a) = \pi^*_h(a|s)\mu_h^*(s)$.
\end{corollary}

\section{Mean-Field Anchored Proximal Gradient over Occupation Measures}
\label{sec:anchored_method}

In this section, we present an anchored first-order method for computing a finite-horizon MFE in the occupation-measure formulation. In general, for monotone inclusions of the form \(\mathbf{0} \in F(\mathbf{z}) + A(\mathbf{z})\), vanilla proximal descent methods are not guaranteed to converge. To address this issue, we employ an acceleration mechanism that ensures convergence.

For any given $\alpha>0$, since \(\mathcal K\) is a nonempty, compact, convex polytope, the normal-cone operator \(N_{\mathcal K}\) is maximal monotone, and its resolvent coincides with the Euclidean projection onto \(\mathcal K\); that is,
\[
J_{\alpha N_{\mathcal K}}=(I+\alpha N_{\mathcal K})^{-1}=\Pi_{\mathcal K}.
\]
Therefore, a proximal anchored-gradient step for \eqref{eq:mi-main} reduces to a projected anchored step on the occupation variable.

\paragraph{Anchored step.}
Starting from an initial feasible occupation flow \(\pmb{\rho^0}\in \mathcal K\), the iteration is given by
\begin{equation}\label{eq:anchored-update}
\pmb{\rho^{k+1}}
=
\Pi_{\mathcal K}
\Bigl(
(1-\beta_k)\pmb{\rho^k} + \beta_k \pmb{\rho^0} - \alpha_k F(\pmb{\rho^k})
\Bigr),
\qquad k=0,1,2,\dots.
\end{equation}
Here, \(\alpha_k > 0\) denotes the step size, and \(\beta_k \in [0,1]\) is the anchoring parameter.
A standard choice is
\begin{equation}\label{eq:stepsizes}
\alpha_k = \frac{1}{L\sqrt{k+\gamma}},
\qquad
\beta_k = \frac{\gamma}{k+\gamma},
\qquad \gamma \ge 2,
\end{equation}
where \(L := |S||A|\,L_r\) is the Lipschitz constant of the operator \(F\), as established in Lemma \ref{F-Lip}.

\begin{remark}
    The step-size choice in \eqref{eq:stepsizes} was first introduced in \cite{surina2026improved} for the unconstrained case, and was later extended to the constrained case in \cite{cai2026last}.
\end{remark}

The update \eqref{eq:anchored-update} admits a simple interpretation. The term
\(-\alpha_k F(\pmb{\rho^k})\) is a forward step driven by the current reward under the induced
mean-field flow, while the term \(\beta_k(\pmb{\rho^0}-\pmb{\rho^k})\) pulls the iterate back toward
the anchor \(\pmb{\rho^0}\), thereby stabilizing the last iterate. The projection \(\Pi_{\mathcal K}\) then
restores feasibility with respect to the occupation-flow constraints. We also note that implementing the update \eqref{eq:anchored-update} requires access to both the reward function and the transition kernel through suitable oracle queries.

A further remark on the projection step in \eqref{eq:anchored-update} is in order. Since \(\mathcal K\) is a compact convex polytope, computing the projection \(\Pi_{\mathcal K}(\pmb{x})\) is equivalent to solving the quadratic program
\begin{equation}\label{eq:proj}
\min_{\pmb{y} \in \mathcal K} \frac{1}{2}\|\pmb{y}-\pmb{x}\|^2_2,
\end{equation}
which can be solved in polynomial time and is therefore tractable. Indeed, first, note that for some fixed matrix $\pmb A$ and vector $\pmb b$, the set $\mathcal K$ can be rewritten as
\[
\mathcal K =\{ \pmb \rho \in \prod_{h=0}^{H-1}\mathbb{R}_+^{S\times A}: \pmb A \pmb \rho =\pmb b\}.
\]
In particular, \eqref{eq:proj} is a convex quadratic program, and when $\pmb A$ and $\pmb b$ are written in rationals, it can be solved exactly in polynomial time \cite{monteiro1989interior}:
\begin{theorem}\label{thrm:gg}
    Let $L$ be the input size of $\pmb x$. Then \eqref{eq:proj} can be solved exactly in $O(H\,|S|\,|A|\,L)$-arithmetic steps.
\end{theorem}
\begin{proof}
    This result directly follows from \cite{monteiro1989interior}.
\end{proof}

\begin{corollary}\label{cor:a}
    For each $k$, let $L_k$ be the input size of each update $(1-\beta_k)\pmb{\rho^k} + \beta_k \pmb{\rho^0} - \alpha_k F(\pmb{\rho^k})$. Then, each step \eqref{eq:anchored-update} can be calculated in at most $O(H\,|S|\,|A|\,C_r + (H\,|S|\,|A|)^3\,L_k)$-arithmetic steps, where $C_r$ is the computational cost of encoding $r(\cdot,\cdot,\cdot)$.
\end{corollary}
\begin{proof}
    Note that evaluating $F(\pmb{\rho^k})$ requires $H\,|S|\,|A|$ scalar reward evaluations. Thus, as $C_r$ is the arithmetic cost of one scalar reward evaluation, $(1-\beta_k)\pmb{\rho^k} + \beta_k \pmb{\rho^0} - \alpha_k F(\pmb{\rho^k})$ can be evaluated in $O(H\,|S|\,|A|\,C_r)$ steps. Combined with the cost of projection in Theorem \ref{thrm:gg}, we obtain the desired $O(H\,|S|\,|A|\,C_r + (H\,|S|\,|A|)^3\,L_k)$ arithmetic complexity.
\end{proof}

In particular, in contrast to previous work, most notably \cite{isobe2024last}, all intermediate steps of our algorithm are fully tractable. A more closely related line of work is \cite{guo2024mf}, which also uses occupation measures and proposes an iterative algorithm. However, that approach relies on the ability to choose a sufficiently good initialization in order to guarantee merely an asymptotic convergence, whereas our method does not require such an assumption.

\paragraph{Recovery of the policy and state-measure flows.}
Given an occupation-flow iterate \(\pmb{\rho^k}\), recall that the induced state flow is defined by
\[
\mu_h^k(s):=\sum_{a\in A}\rho_h^k(s,a),
\qquad h=0,\dots,H-1.
\]
A policy flow \(\pmb{\pi^k}\) can then be recovered from \(\pmb{\rho^k}\) by disintegration:
\begin{equation}\label{eq:policy-recovery}
\pi_h^k(a\mid s)
=
\begin{cases}
\dfrac{\rho_h^k(s,a)}{\mu_h^k(s)}, & \text{if } \mu_h^k(s)>0, \\[2mm]
\text{any element of }\mathcal P(A), & \text{if } \mu_h^k(s)=0.
\end{cases}
\end{equation}
Thus, each iterate \(\pmb{\rho^k}\) naturally defines an approximate equilibrium candidate \((\pmb{\pi^k},\pmb{\mu^k})\). We note, however, that the recovered policy flow need not be unique, since on states with zero mass one may choose any element of \(\mathcal P(A)\).

As noted in \cite{cai2026last}, a natural notion of error metric for the iteration \eqref{eq:anchored-update} is given by the tangential residual
\begin{equation}\label{residual}
r^{\mathrm{tan}}_{F,N_{\mathcal K}}(\pmb{\rho^T})
:=
\inf_{\pmb{c}\in N_{\mathcal K}(\pmb{\rho^T})}
\|F(\pmb{\rho^T})+\pmb{c}\|_2.
\end{equation}
By Corollary \ref{cor:1}, an occupation flow \(\pmb{\rho^*}\) induces an MFE if,
and only if,
\(
\pmb 0 \in F(\pmb{\rho^*})+N_{\mathcal K}(\pmb{\rho^*}).
\)
Equivalently, there exists \(\pmb {c^*} \in N_{\mathcal K}(\pmb{\rho^*})\) such
that
\(
\pmb 0=F(\pmb{\rho^*})+\pmb {c^*}.
\)
Thus, the residual in \eqref{residual} measures the smallest remaining error
after correcting \(F(\pmb{\rho^T})\) by a normal-cone direction to
\(\mathcal K\) at \(\pmb{\rho^T}\). In particular, the residual vanishes exactly
when \(\pmb{\rho^T}\) satisfies the MFE condition.

\begin{theorem}\label{thrm:cc}
Suppose that $G:\mathbb R^d\to\mathbb R^d$ is monotone and $L$-Lipschitz continuous with respect to $\|\cdot\|_2$, and that $A:\mathbb R^d\rightrightarrows\mathbb R^d$ is maximally monotone. Further, suppose that there exists $\pmb z^\star\in\mathbb R^d$ such that $\pmb 0\in G(\pmb z^\star)+A(\pmb z^\star).$
Then, the iterates generated by
$\pmb z_{t+1}=J_{\alpha_t A}\left((1-\beta_t)\pmb z_t+\beta_t\pmb z_0-\alpha_t G(\pmb z_t)\right)$
satisfy
$$r^{\mathrm{tan}}_{G,A}(\pmb z_T)\le\frac{25\gamma L\|\pmb z^\star-\pmb z_0\|_2}{\sqrt{T-1+\gamma}}.$$
\end{theorem}
\begin{proof}
    This result is the same as \cite[Theorem 2]{cai2026last}.
\end{proof}

\begin{theorem}\label{thrm:conv}
    Suppose that Assumptions \ref{weakmon}--\ref{lip} hold. Then, the iterates generated by \eqref{eq:anchored-update} with step sizes given by \eqref{eq:stepsizes} satisfy
    \[
    r^{\mathrm{tan}}_{F,N_{\mathcal K}}(\pmb{\rho^T})
    \le \frac{25\,\gamma \,\sqrt{2H} \,|S| \,|A|\, L_r}{\sqrt{T-1+\gamma}}.
    \]
\end{theorem}

\begin{proof}
    By Lemma \ref{F-mon}, $F$ is monotone on $\mathcal K$ and by Lemma \ref{F-Lip} $F$ is $|S||A|L_r$-Lipschitz on $\mathcal K$. Note that the iterations \eqref{eq:anchored-update} remain on the set $\mathcal K$. Thus, although $F$ is not monotone on the whole $\mathbb R^d$, Theorem \ref{thrm:cc} is still applicable to $F$. By Remark \ref{rem:1} and Lemma \ref{lem:1}, we have that $N_{\mathcal K}$ is maximally monotone. Lastly, by \cite[Proposition 1]{perolat2021scaling} and Assumption \ref{weakmon}, there exists $\pmb 0 \in F(\pmb \rho^*) + N_{\mathcal K}(\pmb {\rho^*})$ for some $\pmb {\rho^*} \in \mathcal K$. Thus, since $\|\pmb {z_t} -\pmb {z_0}\|_2 \le \sqrt{2H}$, the desired rate of convergence follows directly from Theorem \ref{thrm:cc}.
\end{proof}

It remains to verify that every accumulation point of the sequence
\((\pmb{\rho^T})_T\) induces an MFE. Since this sequence is bounded, all of its accumulation points can be reached via some subsequence.

\begin{lemma}\label{lem:cc}
    For any accumulation point \(\pmb{\rho^*}\) of \((\pmb{\rho^T})_T\) we have \(
r^{\mathrm{tan}}_{F,N_{\mathcal K}}(\pmb{\rho^*})=0.
\)
\end{lemma}

Let \(\pmb{\rho^*}\) be an accumulation point of \((\pmb{\rho^T})_T\). Thus, there exists a subsequence \((T_n)_n\) such that
\(
\lim_{n\to\infty} \pmb{\rho^{T_n}}=\pmb{\rho^*}.
\)
By Lemma \ref{lem:cc}, we have
\(
r^{\mathrm{tan}}_{F,N_{\mathcal K}}(\pmb{\rho^*})=0.
\)
Hence, by \eqref{residual}, there exists some
\(\pmb{c^*} \in N_{\mathcal K}(\pmb{\rho}^*)\) such that
\(
\mathbf{0}=F(\pmb{\rho^*})+\pmb{c^*}.
\)
Equivalently,
\(
\mathbf{0}\in F(\pmb{\rho^*})+N_{\mathcal K}(\pmb{\rho^*}).
\)
Therefore, by Corollary \ref{cor:1}, \(\pmb{\rho^*}\) induces a
finite-horizon MFE. Since \(\pmb{\rho^*}\) was chosen as an arbitrary
accumulation point of \((\pmb{\rho^T})_T\), every accumulation point of this
sequence induces an MFE.

The next result identifies the unique limit of $\pmb{\rho^T}$.

\begin{theorem}
Let $\mathcal S\subset \mathcal K$ denote the set of occupancy measures that induce an MFE. If Assumptions \ref{weakmon}--\ref{lip} hold, then $\mathcal S$ is a compact and convex set. Let $\Pi_{\mathcal S}(\pmb{\rho^0})$ denote the orthogonal projection of $\pmb{\rho^0}$ onto $\mathcal S$. Then,
$\lim_{T\to\infty}\pmb{\rho^T}=\Pi_{\mathcal S}(\pmb{\rho^0})$
with respect to the Euclidean norm. In particular, when one has complete access to $r$ and $P$, the iterations in \eqref{eq:anchored-update} are tractable and converge to an occupancy measure that induces an MFE.
\end{theorem}
\begin{proof}
We show in the appendix that $\lim_{T \to \infty} \pmb{\rho^T}=\Pi_{\mathcal S}(\pmb{\rho^0}).$
The convergence of the iterations in \eqref{eq:anchored-update} follows from Theorem \ref{thrm:conv}, Lemma \ref{lem:cc}, and the fact that $\lim_{T \to \infty} \pmb{\rho^T}=\Pi_{\mathcal S}(\pmb{\rho^0}).$
When one has access to $P$, the QP in \eqref{eq:proj} can be solved in polynomial time by Corollary \ref{cor:a}. Thus, the algorithm is fully tractable and converges to an exact MFE.
\end{proof}

Moreover, under our assumptions, one can construct an example for which the rate \(\Omega(1/\sqrt{T})\) is also attained as a matching lower bound for this iteration scheme. In this sense, the bound in Theorem \ref{thrm:conv} is tight for this algorithm.

\begin{proposition}\label{prop:a}
    There exists a mean-field game with time-homogeneous reward function and transition kernel for which the convergence rate in Theorem \ref{thrm:conv} is matched by a corresponding lower bound.
\end{proposition}
\begin{proof}
    We will provide the example here but defer the proof of the lower bound for the rate of convergence of the iterations to the appendix.

    For the proof we consider the finite-horizon MFG with $H=2$, state space
$S=\{0,1,2\}$, action space $A=\{L,R\}$, and initial distribution
$\mu_0=\delta_0$. The transition kernel is independent of the population
measure and satisfies $P(1\mid1,\cdot)\equiv P(2\mid 2,\cdot)=1$ and
\[
P(1\mid 0,L)=P(2\mid 0,R)=1,
\qquad
P(2\mid 0,L)=P(1\mid 0,R)=0,
\]
with states $1$ and $2$ absorbing. The rewards are
\[
r(0,\cdot,\cdot)\equiv 0,
\qquad
r(s,a,\mu)=-\mu(s),
\quad s\in\{1,2\},\ a\in\{L,R\}.
\]
It can be easily shown that this example satisfies the assumptions \ref{weakmon}--\ref{lip}.
\end{proof}

Several remarks are in order concerning Proposition \ref{prop:a}. The upper bound proved in Theorem \ref{thrm:conv} is not optimal in general for other gradient methods in  unconstrained monotone inclusion problems; see \cite{yoon2021accelerated}. In our case, however, since \(N_{\mathcal K} \not\equiv \emptyset\), the problem falls into the constrained setting. For constrained monotone inclusion problems, \cite{cai2022accelerated} shows that, under additional assumptions---which in our framework amount to comonotonicity of the operator \((F,N_{\mathcal K})\)---one can recover the sharper rate established in \cite{yoon2021accelerated}.

\section{Conclusion}

In this work, by formulating the MFE problem as a monotone inclusion problem, we developed algorithms with explicit convergence guarantees for computing MFEs in weakly monotone MFGs. In addition, Proposition \ref{prop:a} shows that the rate \(\Omega(1/\sqrt{T})\) appears as a matching lower bound for our algorithm, implying the tightness of the upper bound.

The main limitation of our analysis is Assumption \ref{ass:2}. An important direction for future work is therefore to relax this assumption. A second interesting direction is to extend the convergence theory of anchored extragradient and proximal methods to monotone inclusion problems involving merely Lipschitz continuous monotone operators together with maximal monotone perturbations. Such optimistic methods can achieve faster convergence rates of order \(O(1/T)\), and are therefore especially attractive for this class of problems.

\bibliography{bibliography}

%%%%%%%%%%%%%%%%%%%%%%%%%%%%%%%%%%%%%%%%%%%%%%%%%%%%%%%%%%%%

\appendix

\section{Detailed Proofs of the Statements in Section \ref{sect:3}}

\begin{lemma}\label{lem:a}
    Every policy flow induces a feasible occupation flow.
\end{lemma}
\begin{proof}
Fix a policy flow $\pmb \pi=(\pi_h)_{h=0}^{H-1}$. Let $\pmb{\mu^\pi}=(\mu_h^{\pmb \pi})_{h=0}^{H-1}$ be the corresponding state flow,
defined recursively by
\[
\mu_1^{\pmb \pi}=\mu_0,\qquad
\mu_{h+1}^{\pmb \pi}(s')=
\sum_{s\in S}\sum_{a\in A}P(s'\mid s,a)\,\pi_h(a\mid s)\,\mu_h^{\pmb \pi}(s)
\quad (h=0,\dots,H-2).
\]
Define an occupation flow associated with $\pmb \pi$ (starting from $\pi_0(a|s)\mu_0(s)$) by
\[
\rho_h^{\pmb \pi}(s,a):=\pi_h(a\mid s)\,\mu_h^{\pmb \pi}(s).
\]
Then for every $s\in S$,
\[
\sum_{a\in A}\rho_0^{\pmb \pi}(s,a)
=
\sum_{a\in A}\pi_0(a\mid s)\,\mu_0^{\pmb \pi}(s)
=
\mu_0^{\pmb \pi}(s)
=
\mu_0(s).
\]
Also, for each $h=1,\dots,H-1$ and $s'\in S$,
\begin{align*}
\sum_{a\in A}\rho_{h+1}^{\pmb \pi}(s',a)
=
\mu_{h+1}^{\pmb \pi}(s')
&=
\sum_{s\in S}\sum_{a\in A}P(s'\mid s,a)\,\pi_h(a\mid s)\,\mu_h^{\pmb \pi}(s)
\\&=
\sum_{s\in S}\sum_{a\in A}P(s'\mid s,a)\,\rho_h^{\pmb \pi}(s,a).
\end{align*}
Hence $\pmb {\rho^\pi}\in \mathcal{K}$.
\end{proof}
\begin{lemma}\label{lem:b}
Every feasible occupation flow comes from some policy flow.
\end{lemma}
\begin{proof}
Fix $\pmb \rho\in \mathcal{K}$ and define
\[
\mu_h^{\pmb \rho}(s):=\sum_{a\in A}\rho_h(s,a).
\]
We define a policy $\pmb{\pi^\rho}=(\pi_h^{\pmb \rho})_{h=0}^{H-1}$ by
\[
\pi_h^{\pmb \rho}(a\mid s):=
\begin{cases}
\dfrac{\rho_h(s,a)}{\mu_h^{\pmb \rho}(s)}, & \text{if }\mu_h^{\pmb \rho}(s)>0,\\[2mm]
\text{an arbitrary element of }\mathcal P(A), & \text{if }\mu_h^{\pmb \rho}(s)=0.
\end{cases}
\]
We claim that the state flow generated by the policy flow $\pmb {\pi^\rho}$ is exactly $\pmb {\mu^\rho}$ and that its occupation
flow is exactly $\pmb \rho$.

We prove by induction on $h$ that $\mu_h^{\pmb {\pi^\rho}}=\mu_h^{\pmb \rho}$.
For $h=0$,
\[
\mu_0^{\pmb {\pi^\rho}}(s)=\mu_0(s)=\sum_{a\in A}\rho_0(s,a)=\mu_0^{\pmb \rho}(s).
\]
Assume that $\mu_k^{\pmb {\pi^\rho}}=\mu_k^{\pmb \rho}$ for all $k=1,\cdots,h$. Then, for each $s'\in S$,
\begin{align*}
\mu_{h+1}^{\pmb {\pi^\rho}}(s')
&=
\sum_{s\in S}\sum_{a\in A}P(s'\mid s,a)\,\pi_h^{\pmb \rho}(a\mid s)\,\mu_h^{\pmb {\pi^\rho}}(s)\\
&=
\sum_{s\in S}\sum_{a\in A}P(s'\mid s,a)\,\pi_h^{\pmb \rho}(a\mid s)\,\mu_h^{\pmb \rho}(s)\\
&=
\sum_{s\in S}\sum_{a\in A}P(s'\mid s,a)\,\rho_h(s,a)\\
&=
\sum_{a\in A}\rho_{h+1}(s',a)
\\&=
\mu_{h+1}^{\pmb \rho}(s').
\end{align*}
By induction, it follows that $\pmb \mu^{\pmb {\pi^\rho}}=\pmb \mu^{\pmb \rho}$.

Finally,
\[
\rho_h^{\pmb {\pi^\rho}}(s,a)=\pi_h^{\pmb \rho}(a\mid s)\mu_h^{\pmb {\pi^\rho}}(s)
=\pi_h^{\pmb \rho}(a\mid s)\mu_h^{\pmb \rho}(s)=\rho_h(s,a),
\]
with the convention that if $\mu_h^{\pmb \rho}(s)=0$, then both sides are zero. Thus every $\pmb \rho\in\mathcal{K}$ is the occupation flow of some policy flow.
\end{proof}

\begin{lemma}\label{lem:6}
$\mathcal{K}$ is a nonempty compact convex polytope.
\end{lemma}
\begin{proof}
By definition, $\mathcal{K}$ is the intersection of finitely many affine hyperplanes and the nonnegative
orthant, hence it is a convex polyhedron.
By Lemma \ref{lem:a} it is nonempty.

To show boundedness, let $\pmb \rho\in \mathcal{K}$. Summing the initial constraints over $s$ gives
\[
\sum_{s\in S}\sum_{a\in A}\rho_0(s,a)=\sum_{s\in S}\mu_0(s)=1.
\]
Now assume $\sum_{s,a}\rho_h(s,a)=1$. Then,
\begin{align*}
\sum_{s'\in S}\sum_{a\in A}\rho_{h+1}(s',a)
&=
\sum_{s'\in S}\sum_{s\in S}\sum_{a\in A}P(s'\mid s,a)\rho_h(s,a)\\
&=
\sum_{s\in S}\sum_{a\in A}\rho_h(s,a)\sum_{s'\in S}P(s'\mid s,a)\\
&=
\sum_{s\in S}\sum_{a\in A}\rho_h(s,a)
\\&=1.
\end{align*}
Hence each stage occupation vector has total mass $1$, and thus each coordinate belongs to $[0,1]$.
Therefore $\mathcal{K}$ is bounded.
Being a closed and bounded polyhedron, it is compact.
Thus $\mathcal{K}$ is a nonempty compact convex polytope.
\end{proof}
\begin{lemma}
    Best response at fixed mean field is a linear program on $\mathcal{K}$
\end{lemma}
\begin{proof}
Fix an exogenous mean-field flow $\pmb\mu=(\mu_h)_{h=0}^{H-1}$.
For a policy flow $\pmb\pi=(\pi_h)_{h=0}^{H-1}$, let
$\pmb\mu^{\pmb\pi}=(\mu_h^{\pmb\pi})_{h=0}^{H-1}$ denote the state-flow induced by
$\pmb\pi$ under the transition kernel $P(\cdot\mid s,a)$ and the initial distribution
$\mu_0$. Define a corresponding occupation flow
$\pmb\rho^{\pmb\pi}=(\rho_h^{\pmb\pi})_{h=0}^{H-1}$ by
\[
\rho_h^{\pmb\pi}(s,a):=\pi_h(a\mid s)\mu_h^{\pmb\pi}(s).
\]
Then the best-response objective at the exogenous flow $\pmb\mu$ is
\[
J(\pmb\mu,\pmb\pi)
:=
\mathbb E^{\pmb\pi}\!\left[\sum_{h=0}^{H-1} r(x_h,a_h,\mu_h)\right],
\]
and therefore by Assumption \ref{ass:2} we obtain
\[
J(\pmb\mu,\pmb\pi)
=
\sum_{h=0}^{H-1}\sum_{s\in S}\sum_{a\in A}
\rho_h^{\pmb\pi}(s,a)\,r(s,a,\mu_h).
\]

Therefore, for a fixed exogenous mean-field flow $\pmb\mu=(\mu_h)_{h=0}^{H-1}$,
maximizing $J(\pmb\mu,\pmb\pi)$ over policy flows $\pmb\pi$ is equivalent to maximizing
the linear functional
\[
L_{\pmb\mu}(\pmb\rho)
:=
\sum_{h=0}^{H-1}\sum_{s\in S}\sum_{a\in A}
\rho_h(s,a)\,r(s,a,\mu_h)
\]
over $\pmb\rho=(\rho_h)_{h=0}^{H-1}\in \mathcal K$.
Indeed, by Lemma~\ref{lem:a}, every policy flow $\pmb\pi$ induces an occupation flow
$\pmb\rho^{\pmb\pi}\in\mathcal K$, and by Lemma~\ref{lem:b}, every
$\pmb\rho\in\mathcal K$ is the occupation flow of some policy flow. Moreover, if
$\pmb\rho^{\pmb\pi}$ denotes the occupation flow induced by $\pmb\pi$, then
\[
J(\pmb\mu,\pmb\pi)
=
\sum_{h=0}^{H-1}\sum_{s\in S}\sum_{a\in A}
\rho_h^{\pmb\pi}(s,a)\,r(s,a,\mu_h)
=
L_{\pmb\mu}(\pmb\rho^{\pmb\pi}).
\]
Hence maximizing $J(\pmb\mu,\pmb\pi)$ over policy flows is equivalent to maximizing
$L_{\pmb\mu}(\pmb\rho)$ over $\pmb\rho\in\mathcal K$.
\end{proof}
\begin{lemma}
Mean-field equilibria are exactly solutions of the variational inequality \eqref{eq:vi-main}.
\end{lemma}

\begin{proof}
Suppose first that $(\pmb{\mu^\star},\pmb{\pi^\star})$ is an MFE, and let
$\pmb{\rho^\star}:=\pmb{\rho^{\pi^\star}}$ be the occupation flow induced by $\pmb{\pi^\star}$.
Since $(\pmb{\mu^\star},\pmb{\pi^\star})$ is an MFE, $\pmb{\pi^\star}$ is a
best response to the exogenous mean-field flow $\pmb{\mu^\star}$, that is,
\[
J(\pmb{\mu^\star},\pmb{\pi})
\le
J(\pmb{\mu^\star},\pmb{\pi^\star})
\qquad \forall \pmb{\pi}.
\]
Let $\pmb{\rho}\in\mathcal K$ be arbitrary. By Lemma~\ref{lem:b}, there exists a policy flow
$\pmb{\pi^\rho}$ whose occupation flow is exactly $\pmb{\rho}$. By Lemma~\ref{lem:6},
\[
J(\pmb{\mu^\star},\pmb{\pi^\rho})
=
L_{\pmb{\mu^\star}}(\pmb{\rho}),
\qquad
J(\pmb{\mu^\star},\pmb{\pi^\star})
=
L_{\pmb{\mu^\star}}(\pmb{\rho^\star}),
\]
where
\[
L_{\pmb{\mu^\star}}(\pmb{\rho})
:=
\sum_{h=0}^{H-1}\sum_{s\in S}\sum_{a\in A}
\rho_h(s,a)\,r(s,a,\mu_h^\star).
\]
Since $(\pmb{\mu^\star},\pmb{\pi^\star})$ is an MFE,
\[
L_{\pmb{\mu^\star}}(\pmb{\rho})
=
J(\pmb{\mu^\star},\pmb{\pi^\rho})
\le
J(\pmb{\mu^\star},\pmb{\pi^\star})
=
L_{\pmb{\mu^\star}}(\pmb{\rho^\star})
\qquad \forall \pmb{\rho}\in\mathcal K.
\]
Equivalently,
\[
\sum_{h=0}^{H-1}\sum_{s\in S}\sum_{a\in A}
(\rho_h(s,a)-\rho_h^\star(s,a))\,r(s,a,\mu_h^\star)\le 0
\qquad \forall \pmb{\rho}\in\mathcal K.
\]
Since $(\pmb{\mu^\star},\pmb{\pi^\star})$ is an MFE, the consistency condition
gives
\[
\pmb{\mu^\star}=\pmb{\mu^{\pi^\star}}.
\]
Moreover, since $\pmb{\rho^\star}$ is the occupation flow of $\pmb{\pi^\star}$, we also have
\[
\pmb{\mu^{\pi^\star}}=\pmb{\mu^{\rho^\star}}.
\]
Hence
\[
\pmb{\mu^\star}=\pmb{\mu^{\rho^\star}},
\]
and thus the previous inequality becomes
\[
\sum_{h=0}^{H-1}\sum_{s\in S}\sum_{a\in A}
(\rho_h(s,a)-\rho_h^\star(s,a))\,r(s,a,\mu_h^{\pmb{\rho^\star}})\le 0
\qquad \forall \pmb{\rho}\in\mathcal K.
\]
Rearranging,
\[
\sum_{h=0}^{H-1}\sum_{s\in S}\sum_{a\in A}
\bigl[-r(s,a,\mu_h^{\pmb{\rho^\star}})\bigr]
(\rho_h(s,a)-\rho_h^\star(s,a))
\ge 0
\qquad \forall \pmb{\rho}\in\mathcal K.
\]
By the definition of $F$, this is exactly
\[
\langle F(\pmb{\rho^\star}),\pmb{\rho}-\pmb{\rho^\star}\rangle \ge 0
\qquad \forall \pmb{\rho}\in\mathcal K.
\]
Hence $\pmb{\rho^\star}$ solves the variational inequality \eqref{eq:vi-main}.

Conversely, suppose that $\pmb{\rho^\star}\in\mathcal K$ solves the variational inequality
\[
\langle F(\pmb{\rho^\star}),\pmb{\rho}-\pmb{\rho^\star}\rangle \ge 0
\qquad \forall \pmb{\rho}\in\mathcal K.
\]
By the definition of $F$, this means that
\[
\sum_{h=0}^{H-1}\sum_{s\in S}\sum_{a\in A}
(\rho_h(s,a)-\rho_h^\star(s,a))\,r(s,a,\mu_h^{\pmb{\rho^\star}})\le 0
\qquad \forall \pmb{\rho}\in\mathcal K.
\]
Define
\[
\pmb{\mu^\star}:=\pmb{\mu^{\rho^\star}},
\qquad
\pmb{\pi^\star}:=\pmb{\pi^{\rho^\star}},
\]
where $\pmb{\pi^{\rho^\star}}$ is given by Lemma~\ref{lem:b} (the exact choice does not matter). Then
\[
\sum_{h=0}^{H-1}\sum_{s\in S}\sum_{a\in A}
(\rho_h(s,a)-\rho_h^\star(s,a))\,r(s,a,\mu_h^\star)\le 0
\qquad \forall \pmb{\rho}\in\mathcal K,
\]
that is,
\[
L_{\pmb{\mu^\star}}(\pmb{\rho})
\le
L_{\pmb{\mu^\star}}(\pmb{\rho^\star})
\qquad \forall \pmb{\rho}\in\mathcal K.
\]
Now let $\pmb{\pi}$ be any policy flow, and let $\pmb{\rho^\pi}$ be a corresponding occupation flow.
By Lemma~\ref{lem:a}, we have $\pmb{\rho^\pi}\in\mathcal K$. Therefore,
\[
L_{\pmb{\mu^\star}}(\pmb{\rho^\pi})
\le
L_{\pmb{\mu^\star}}(\pmb{\rho^\star}).
\]
Using Lemma~\ref{lem:6} again,
\[
J(\pmb{\mu^\star},\pmb{\pi})
=
L_{\pmb{\mu^\star}}(\pmb{\rho^\pi})
\le
L_{\pmb{\mu^\star}}(\pmb{\rho^\star})
=
J(\pmb{\mu^\star},\pmb{\pi^\star}).
\]
Hence $\pmb{\pi^\star}$ is a best response to $\pmb{\mu^\star}$.

Finally, by construction,
\[
\pmb{\mu^\star}=\pmb{\mu^{\rho^\star}}
\qquad\text{and}\qquad
\pmb{\rho^\star}=\pmb{\rho^{\pi^\star}},
\]
and thus
\[
\pmb{\mu^\star}=\pmb{\mu^{\pi^\star}}.
\]
Thus the consistency condition also holds, and therefore
$(\pmb{\pi^\star},\pmb{\mu^\star})$ is a mean-field equilibrium.
\end{proof}
\begin{lemma}
F is monotone on \(\mathcal K\).
\end{lemma}
\begin{proof}
Take any $\pmb{\rho},\pmb{\tilde\rho}\in\mathcal K$. By the definition of $F$,
\begin{align*}
\langle F(\pmb{\rho})-F(\pmb{\tilde\rho}),\pmb{\rho}-\pmb{\tilde\rho}\rangle
&=
\sum_{h=0}^{H-1}\sum_{s\in S}\sum_{a\in A}
\Bigl(
-r(s,a,\mu_h^{\pmb{\rho}})
+r(s,a,\mu_h^{\pmb{\tilde\rho}})
\Bigr)
\bigl(\rho_h(s,a)-\tilde\rho_h(s,a)\bigr) \\
&=
-
\sum_{h=0}^{H-1}\sum_{s\in S}\sum_{a\in A}
\Bigl(
r(s,a,\mu_h^{\pmb{\rho}})
-r(s,a,\mu_h^{\pmb{\tilde\rho}})
\Bigr)
\bigl(\rho_h(s,a)-\tilde\rho_h(s,a)\bigr).
\end{align*}

By Lemma~\ref{lem:b}, there exist policy flows
$\pmb{\pi}^{\pmb{\rho}}$ and $\pmb{\pi}^{\pmb{\tilde\rho}}$
whose occupation flows are exactly $\pmb{\rho}$ and $\pmb{\tilde\rho}$, respectively.
Moreover, by construction, the induced state flows of these policies are
$\pmb{\mu}^{\pmb{\rho}}$ and $\pmb{\mu}^{\pmb{\tilde\rho}}$.

Applying Assumption~\ref{weakmon} to the two policy flows
$\pmb{\pi}^{\pmb{\rho}}$ and $\pmb{\pi}^{\pmb{\tilde\rho}}$, we obtain
\[
\sum_{h=0}^{H-1}\sum_{s\in S}\sum_{a\in A}
\Bigl(
r(s,a,\mu_h^{\pmb{\rho}})
-r(s,a,\mu_h^{\pmb{\tilde\rho}})
\Bigr)
\bigl(\rho_h(s,a)-\tilde\rho_h(s,a)\bigr)
\le 0.
\]
Therefore,
\[
\langle F(\pmb{\rho})-F(\pmb{\tilde\rho}),\pmb{\rho}-\pmb{\tilde\rho}\rangle
\ge 0.
\]
Hence $F$ is monotone on $\mathcal K$.
\end{proof}
\begin{lemma}
$F$ is Lipschitz continuous.
\end{lemma}
\begin{proof}
Let $N:=$.
For any given $h,s,a$, Assumption \ref{lip} implies
\[
|F_{h,s,a}(\pmb \rho)-F_{h,s,a}(\pmb {\tilde\rho})|
=
|r(s,a,\mu_h^{\pmb \rho})-r(s,a,\mu_h^{\pmb {\tilde\rho}})|
\le
L_r\|\mu_h^{\pmb \rho}-\mu_h^{\pmb{\tilde\rho}}\|_1.
\]
Hence
\[
\sum_{s\in S}\sum_{a\in A}|F_{h,s,a}(\pmb \rho)-F_{h,s,a}(\pmb{\tilde\rho})|^2
\le
|S|^2|A|^2 L_r^2 \|\mu_h^{\pmb \rho}-\mu_h^{\pmb{\tilde\rho}}\|_1^2.
\]
Note that
\begin{align*}
\|\mu_h^{\pmb \rho}-\mu_h^{\pmb{\tilde\rho}}\|_1
&=
\sum_{s\in S}\left|\sum_{a\in A}\bigl(\rho_h(s,a)-\tilde\rho_h(s,a)\bigr)\right|\\
&\le
\sum_{s\in S}\sum_{a\in A}|\rho_h(s,a)-\tilde\rho_h(s,a)|\\
&=
\|\rho_h-\tilde\rho_h\|_1
\\&\le
|S|\,|A|\,\|\rho_h-\tilde\rho_h\|_2.
\end{align*}
Therefore
\[
\sum_{s\in S}\sum_{a\in A}|F_{h,s,a}(\pmb \rho)-F_{h,s,a}(\pmb{\tilde\rho})|^2
\le
|S|^2\,|A|^2 \,L_r^2 \|\rho_h-\tilde\rho_h\|_2^2.
\]
Summing over $h=0,\dots,H-1$, we obtain
\[
\|F(\pmb \rho)-F(\pmb{\tilde\rho})\|_2^2
\le
|S|^2|A|^2 L_r^2 \|\pmb \rho-\pmb{\tilde\rho}\|_2^2.
\]
Taking square roots yields
\[
\|F(\pmb \rho)-F(\pmb{\tilde\rho})\|_2
\le
|S||A|\,L_r\,\|\pmb\rho-\pmb{\tilde\rho}\|_2.
\]
Thus, $F$ is a Lipschitz continuous functional, as desired.
\end{proof}
\begin{lemma}
\eqref{eq:vi-main} and monotone inclusion are equivalent.
\end{lemma}
\begin{proof}
Recall that for a closed convex set $\mathcal{K}$, the normal cone at $\pmb{\rho^\star}\in \mathcal{K}$ is
\[
N_{\mathcal{K}}(\pmb{\rho^\star}):=
\left\{
\pmb g\in \mathbb{R}^{H|S||A|}:
\langle \pmb g,\pmb \rho-\pmb \rho^\star\rangle \le 0\quad \forall \pmb \rho\in \mathcal{K}
\right\}.
\]
Thus,
\[
\langle F(\pmb{\rho^\star}),\pmb \rho-\pmb{\rho^\star}\rangle \ge 0
\quad \forall \pmb \rho\in \mathcal{K}
\]
is equivalent to
\[
-F(\pmb{\rho^\star})\in N_{\mathcal{K}}(\pmb{\rho^\star}),
\]
which is equivalent to
\[
\pmb 0\in F(\pmb{\rho^\star})+N_{\mathcal{K}}(\pmb{\rho^\star}).
\]
This proves \eqref{eq:mi-main} and completes the proof.
\end{proof}

\begin{corollary}
Under the assumptions of the theorem, the finite-horizon MFE problem can be solved in the occupation
variable by algorithms for monotone inclusions. In particular, since $F$ is monotone and Lipschitz and
$N_{\mathcal{K}}$ is maximal monotone, the proximal anchored gradient method applies to
\[
\pmb 0\in F(\pmb \rho)+N_{\mathcal{K}}(\pmb \rho).
\]
\end{corollary}
\section{Further Details of Section \ref{sec:anchored_method}}

\begin{lemma}
   The graph of $N_{\mathcal K}$ is closed.
\end{lemma}
\begin{proof}
    By Lemma \ref{lem:1}, $\mathcal K$ is compact and convex. It follows that the graph of $N_{\mathcal K}$ is closed by \cite[Theorem~24.4]{rockafellar1970convex} and Remark \ref{rem:1}.
\end{proof}

In what follows, we focus on the proof of Lemma \ref{lem:cc}.

\begin{proof}[Proof of Lemma \ref{lem:cc}]
    It remains to verify that every accumulation point of the sequence
\((\pmb{\rho^T})_T\) induces an MFE. The only missing detail from Section \ref{sec:anchored_method} is that any accumulation point $\pmb{\rho^*}$ of $(\pmb{\rho^T})_T$ satisfies 
\[
r^{\mathrm{tan}}_{F,N_{\mathcal K}}(\pmb{\rho^*})=0.
\]
Let \((\pmb{\rho^{T_n}})_n\) be a convergent subsequence such that
\[
\pmb{\rho^{T_n}}\to \pmb{\rho^*} .
\]
Since \(\mathcal K\) is compact, it is closed. Hence
\(\pmb{\rho^*}\in \mathcal K\).

By Theorem \ref{thrm:conv} we have
\[
r^{\mathrm{tan}}_{F,N_{\mathcal K}}(\pmb{\rho^{T_n}})\to 0 .
\]
By the definition of the tangent residual, for each sufficiently large \(n\) there exists
\[
\pmb {c_n}\in N_{\mathcal K}(\pmb{\rho^{T_n}})
\]
such that
\[
\|F(\pmb{\rho^{T_n}})+\pmb {c_n}\|_2
\le
r^{\mathrm{tan}}_{F,N_{\mathcal K}}(\pmb{\rho^{T_n}})
+\frac{1}{n}.
\]
Therefore,
\[
F(\pmb{\rho^{T_n}})+\pmb {c_n} \to \mathbf 0 .
\]
Equivalently,
\[
\pmb {c_n} = -F(\pmb{\rho^{T_n}}) + o(1).
\]
Since \(F\) is continuous and \(\pmb{\rho^{T_n}}\to \pmb{\rho^*}\), the sequence
\((F(\pmb{\rho^{T_n}}))_n\) is bounded. Hence \((\pmb {c_n})_n\) is also bounded.
Passing to a further subsequence if necessary, we may assume that
\[
\pmb {c_n} \to \pmb {c^*}
\]
for some vector \(\pmb {c^*}\).

The normal cone mapping \(N_{\mathcal K}\) has a closed graph because
\(\mathcal K\) is closed and convex. Since
\[
\pmb{\rho^{T_n}}\to \pmb{\rho^*},
\qquad
\pmb {c_n}\to \pmb {c^*},
\qquad
\pmb {c_n}\in N_{\mathcal K}(\pmb{\rho^{T_n}}),
\]
we obtain
\[
\pmb {c^*}\in N_{\mathcal K}(\pmb{\rho}^*).
\]
Moreover, by continuity of \(F\),
\[
F(\pmb{\rho^{T_n}})\to F(\pmb{\rho^*}).
\]
Taking limits in
\[
F(\pmb{\rho^{T_n}})+\pmb {c_n} \to \mathbf 0
\]
gives
\[
F(\pmb{\rho^*})+\pmb {c^*}=\mathbf 0.
\]
Since \(\pmb {c^*}\in N_{\mathcal K}(\pmb{\rho}^*)\), it follows that
\[
\mathbf 0\in F(\pmb{\rho^*})+N_{\mathcal K}(\pmb{\rho^*}).
\]
Therefore,
\[
r^{\mathrm{tan}}_{F,N_{\mathcal K}}(\pmb{\rho^*})
=
\inf_{\pmb c\in N_{\mathcal K}(\pmb{\rho^*})}
\|F(\pmb{\rho}^*)+\pmb c\|_2
=0,
\]
as desired.
\end{proof}

{
Next, we show that $\pmb{\rho^T} \to \Pi_{S}(\pmb{\rho^0})$ under the Euclidean norm. First, we demonstrate that occupancy measures that induce an MFE form a convex and compact set.}

\begin{lemma}
\label{lem:solution-set-closed-convex}
The following statements hold:
\begin{itemize}
    \item The set of occupation flows that induce an MFE,
    \[
    \mathcal S
    :=
    \left\{
    \pmb{\rho}\in\mathcal K:
    \langle F(\pmb{\rho}),\pmb z-\pmb{\rho}\rangle\ge0
    \quad\forall \pmb z\in\mathcal K
    \right\},
    \]
    is closed and convex.

    \item For every anchor term
    \(\pmb{\rho^0}\in \mathcal K\), there exists a unique
    orthogonal projection
    $\Pi_{\mathcal S}(\pmb{\rho^0})$.
\end{itemize}
\end{lemma}

\begin{proof}
Note that by Corollary~\ref{cor:1}, the occupancy measures that induce an MFE belong to $\mathcal S$. Furthermore, every element of $\mathcal S$ induces an MFE too.

To prove that $\mathcal S$ is closed and convex, we provide an equivalent formulation for $\mathcal S$ using monotonicity. We claim that \(\mathcal S\) is equivalent to
\begin{equation}
\label{eq:minty-solution-set}
\overline{\mathcal S}
=
\left\{
\pmb{\rho}\in\mathcal K:
\langle F(\pmb z),\pmb z-\pmb{\rho}\rangle\ge0
\quad
\forall \pmb z\in\mathcal K
\right\}.
\end{equation}

First, suppose that \(\pmb{\rho}\in\mathcal S\). Then, for every
\(\pmb z\in\mathcal K\),
\[
\langle F(\pmb{\rho}),\pmb z-\pmb{\rho}\rangle\ge0.
\]
By monotonicity of \(F\),
\[
\langle F(\pmb z)-F(\pmb{\rho}),
\pmb z-\pmb{\rho}\rangle\ge0.
\]
Adding these inequalities gives
\[
\langle F(\pmb z),\pmb z-\pmb{\rho}\rangle\ge0
\qquad
\forall \pmb z\in\mathcal K.
\]

Conversely, suppose that \(\pmb{\rho}\in\mathcal K\) satisfies
\[
\langle F(\pmb z),\pmb z-\pmb{\rho}\rangle\ge0
\qquad
\forall \pmb z\in\mathcal K.
\]
Fix an arbitrary \(\pmb y\in\mathcal K\), and for \(t\in(0,1]\), define
\[
\pmb z_t
:=
(1-t)\pmb{\rho}+t\pmb y.
\]
Since \(\mathcal K\) is convex, \(\pmb z_t\in\mathcal K\). Therefore,
\[
\langle F(\pmb z_t),\pmb z_t-\pmb{\rho}\rangle\ge0.
\]
Since
\[
\pmb z_t-\pmb{\rho}
=
t(\pmb y-\pmb{\rho}),
\]
we obtain
\[
\langle F(\pmb z_t),\pmb y-\pmb{\rho}\rangle\ge0.
\]
Letting \(t\downarrow0\), continuity of \(F\) yields
\[
\langle F(\pmb{\rho}),\pmb y-\pmb{\rho}\rangle\ge0.
\]
Since \(\pmb y\in\mathcal K\) was arbitrary, it follows that
\(\pmb{\rho}\in\mathcal S\). This proves
\eqref{eq:minty-solution-set}.

For each fixed \(\pmb z\in\mathcal K\), define
\[
H_{\pmb z}
:=
\left\{
\pmb{\rho} \in \mathbb R^{H|S||A|}:
\langle F(\pmb z),\pmb z-\pmb{\rho}\rangle\ge0
\right\}.
\]
The set \(H_{\pmb z}\) is a closed half-space in the variable
\(\pmb{\rho}\), and is therefore closed and convex. By
\eqref{eq:minty-solution-set},
\[
\mathcal S
=
\mathcal K
\cap
\bigcap_{\pmb z\in\mathcal K}H_{\pmb z}.
\]
Since \(\mathcal K\) is closed and convex, and arbitrary intersections
of closed convex sets are closed and convex, it follows that
\(\mathcal S\) is closed and convex.

Now, under monotonicity and Lipschitz continuity, we have \(\mathcal S\neq\varnothing\). Since \(\mathcal S\) is
closed and is contained in the compact set \(\mathcal K\), it is
compact. Thus, existence and uniqueness of
\(\Pi_{\mathcal S}(\pmb{\rho^0})\) follows from the Hilbert projection theorem.
\end{proof}

\begin{theorem}
\label{thrm:anchor-projection-convergence}
Let \((\pmb{\rho^k})_{k\ge0}\) be generated by
\[
\pmb{\rho^{k+1}}
=
\Pi_{\mathcal K}
\left(
(1-\beta_k)\pmb{\rho^k}
+
\beta_k\pmb{\rho^0}
-
\alpha_kF(\pmb{\rho^k})
\right),
\]
where
\[
\alpha_k
=
\frac{1}{L\sqrt{k+\gamma}},
\qquad
\beta_k
=
\frac{\gamma}{k+\gamma},
\qquad
\gamma\ge2.
\]
Then
\[
\pmb{\rho^k}
\longrightarrow
\Pi_{\mathcal S}(\pmb{\rho^0}).
\]
Consequently, the sequence has a unique accumulation point.
\end{theorem}

\begin{proof}
By \cite[Theorem~2]{cai2026last}, we have
\[
r^{\mathrm{tan}}_{F,N_{\mathcal K}}(\pmb{\rho^k})
\longrightarrow0.
\]
We also demonstrated that every accumulation point of
\((\pmb{\rho^k})_{k\ge0}\) therefore belongs to \(\mathcal S\).
By Lemma~\ref{lem:solution-set-closed-convex}, the set
\(\mathcal S\) is nonempty, closed, and convex. Hence
$\pmb p:=\Pi_{\mathcal S}(\pmb{\rho^0})$
exists and is unique. We show that
$\pmb{\rho^k}\to\pmb p.$

Since \(\pmb p\in\mathcal S\), we have $-F(\pmb p)\in N_{\mathcal K}(\pmb p),$
and hence, for every \(k\),
$\pmb p=\Pi_{\mathcal K}\bigl(\pmb p-\alpha_kF(\pmb p)\bigr).$
Using this identity, the nonexpansiveness of
\(\Pi_{\mathcal K}\), and the same monotonicity--Lipschitz estimate
used in the proof of \cite[Lemma~1]{cai2026last}, we show that
\begin{align}
\|\pmb{\rho^{k+1}}-\pmb p\|_2^2
&\le
(1-\theta_k)
\|\pmb{\rho^k}-\pmb p\|_2^2
\nonumber
+
2\beta_k(1-\beta_k)
\left\langle
\pmb{\rho^0}-\pmb p,
\pmb{\rho^k}-\pmb p
\right\rangle
+
\varepsilon_k,
\label{eq:anchor-selection-short}
\end{align}
where
\[
\theta_k
:=
2\beta_k-\beta_k^2-\alpha_k^2L^2
\]
and
\[
0\le\varepsilon_k
\le
O\bigl(\beta_k^2+\alpha_k\beta_k\bigr).
\]
Indeed, since \(\pmb p\in\mathcal S\), we have
$\mathbf 0\in F(\pmb p)+N_{\mathcal K}(\pmb p),$
and therefore $-F(\pmb p)\in N_{\mathcal K}(\pmb p).$
Since \(N_{\mathcal K}(\pmb p)\) is a cone, it follows that
$-\alpha_kF(\pmb p)\in N_{\mathcal K}(\pmb p).$
By the characterization of the Euclidean projection,
$\pmb p=\Pi_{\mathcal K}\bigl(\pmb p-\alpha_kF(\pmb p)\bigr).$

Using the nonexpansiveness of \(\Pi_{\mathcal K}\), we obtain
\begin{align*}
\|\pmb{\rho^{k+1}}-\pmb p\|_2^2
&\le
\left\|
(1-\beta_k)(\pmb{\rho^k}-\pmb p)
+
\beta_k(\pmb{\rho^0}-\pmb p)
-
\alpha_k(F(\pmb{\rho^k})-F(\pmb p))
\right\|_2^2
\\
&=
(1-\beta_k)^2\|\pmb{\rho^k}-\pmb p\|_2^2
+
\beta_k^2\|\pmb{\rho^0}-\pmb p\|_2^2
+
\alpha_k^2\|F(\pmb{\rho^k})-F(\pmb p)\|_2^2
\\
&\quad
+
2\beta_k(1-\beta_k)
\langle \pmb{\rho^0}-\pmb p,\pmb{\rho^k}-\pmb p\rangle
\\
&\quad
-
2\alpha_k(1-\beta_k)
\langle (F(\pmb{\rho^k})-F(\pmb p)),\pmb{\rho^k}-\pmb p\rangle
-
2\alpha_k\beta_k
\langle (F(\pmb{\rho^k})-F(\pmb p)),\pmb{\rho^0}-\pmb p\rangle.
\end{align*}

By monotonicity of \(F\),
\[
\langle (F(\pmb{\rho^k})-F(\pmb p)),(\pmb{\rho^k}-\pmb p)\rangle
=
\left\langle
F(\pmb{\rho^k})-F(\pmb p),
\pmb{\rho^k}-\pmb p
\right\rangle
\ge0 \text{ and hence }
-2\alpha_k(1-\beta_k)
\langle (F(\pmb{\rho^k})-F(\pmb p)),(\pmb{\rho^k}-\pmb p)\rangle
\le0.
\]
Moreover, since \(F\) is \(L\)-Lipschitz,
$\|(F(\pmb{\rho^k})-F(\pmb p))\|_2\le L\|(\pmb{\rho^k}-\pmb p)\|_2.$
Since \(\mathcal K\) is compact and
\(\pmb{\rho^k},\pmb p\in\mathcal K\), there exists \(D<\infty\) such
that
\[
\|(\pmb{\rho^k}-\pmb p)\|_2\le D
\qquad
\forall k\ge0.
\]
Therefore,
\begin{align*}
-2\alpha_k\beta_k
\langle (F(\pmb{\rho^k})-F(\pmb p)),(\pmb{\rho^0}-\pmb p)\rangle
&\le
2\alpha_k\beta_k
\|(F(\pmb{\rho^k})-F(\pmb p))\|_2\|(\pmb{\rho^0}-\pmb p)\|_2
\\
&\le
2\alpha_k\beta_kLD\|(\pmb{\rho^0}-\pmb p)\|_2.
\end{align*}

Combining the preceding estimates gives
\begin{align}
\|\pmb{\rho^{k+1}}-\pmb p\|_2^2
&\le
\left(
(1-\beta_k)^2+\alpha_k^2L^2
\right)
\|\pmb{\rho^k}-\pmb p\|_2^2
+
2\beta_k(1-\beta_k)
\left\langle
\pmb{\rho^0}-\pmb p,
\pmb{\rho^k}-\pmb p
\right\rangle
+
\varepsilon_k,
\label{eq:anchor-selection-short}
\end{align}
where
\[
\varepsilon_k
:=
\beta_k^2
\|\pmb{\rho^0}-\pmb p\|_2^2
+
2\alpha_k\beta_kLD
\|\pmb{\rho^0}-\pmb p\|_2.
\]
Finally, defining
\[
\theta_k
:=
1-
\left(
(1-\beta_k)^2+\alpha_k^2L^2
\right)
=
2\beta_k-\beta_k^2-\alpha_k^2L^2,
\]
we may rewrite \eqref{eq:anchor-selection-short} as
\begin{align}
\|\pmb{\rho^{k+1}}-\pmb p\|_2^2
&\le
(1-\theta_k)
\|\pmb{\rho^k}-\pmb p\|_2^2
+
2\beta_k(1-\beta_k)
\left\langle
\pmb{\rho^0}-\pmb p,
\pmb{\rho^k}-\pmb p
\right\rangle
+
\varepsilon_k.
\end{align}

For completeness, expanding
\[
\left\|
(1-\beta_k)(\pmb{\rho^k}-\pmb p)
+
\beta_k(\pmb{\rho^0}-\pmb p)
-
\alpha_k(F(\pmb{\rho^k})-F(\pmb p))
\right\|_2^2.
\]
The term
\[
-2\alpha_k(1-\beta_k)
\langle (F(\pmb{\rho^k})-F(\pmb p)),(\pmb{\rho^k}-\pmb p)\rangle
\]
is nonpositive by monotonicity, while Lipschitz continuity and the
compactness of \(\mathcal K\) bound all remaining error terms by
\(O(\beta_k^2+\alpha_k\beta_k)\).

For the prescribed step sizes,
\[
\theta_k
=
\frac{2\gamma-1}{k+\gamma}
-
\frac{\gamma^2}{(k+\gamma)^2}.
\]
Hence
\[
0<\theta_k<1,
\qquad
\sum_{k=0}^{\infty}\theta_k=\infty,
\]
and
\[
\frac{\varepsilon_k}{\theta_k}\longrightarrow0,
\qquad
\frac{2\beta_k(1-\beta_k)}{\theta_k}
\longrightarrow
\frac{2\gamma}{2\gamma-1}.
\]

We next claim that
\[
\limsup_{k\to\infty}
\left\langle
\pmb{\rho^0}-\pmb p,
\pmb{\rho^k}-\pmb p
\right\rangle
\le0.
\]
Indeed, choose a subsequence attaining this limsup. Since
\(\mathcal K\) is compact, it has a further subsequence of \(\pmb {\rho^k}\) converging to
some accumulation point \(\pmb{\bar\rho}\). As shown above,
\(\pmb{\bar\rho}\in\mathcal S\). The characterization of the
projection
\[
\pmb p=\Pi_{\mathcal S}(\pmb{\rho^0})
\]
then gives
\[
\left\langle
\pmb{\rho^0}-\pmb p,
\pmb{\bar\rho}-\pmb p
\right\rangle
\le0,
\]
which proves the claim.

The preceding estimates allow
\eqref{eq:anchor-selection-short} to be written as
\[
\|\pmb{\rho^{k+1}}-\pmb p\|_2^2
\le
(1-\theta_k)\|\pmb{\rho^k}-\pmb p\|_2^2+\theta_k b_k,
\]
where $\limsup_{k\to\infty}b_k\le0.$
Since \(a_k\ge0\), \(0<\theta_k<1\), and
\(\sum_k\theta_k=\infty\).

The preceding estimates allow
\eqref{eq:anchor-selection-short} to be written as
\[
a_{k+1}
\le
(1-\theta_k)a_k+\theta_k b_k,
\qquad
a_k:=\|\pmb{\rho^k}-\pmb p\|_2^2,
\]
where
\[
\limsup_{k\to\infty}b_k\le0.
\]
We now show directly that \(a_k\to0\). Fix \(\varepsilon>0\).
Since \(\limsup_{k\to\infty}b_k\le0\), there exists \(K_\varepsilon\)
such that
\[
b_k\le\varepsilon
\qquad
\text{for every }k\ge K_\varepsilon.
\]
Consequently,
\[
a_{k+1}-\varepsilon
\le
(1-\theta_k)(a_k-\varepsilon),
\qquad
k\ge K_\varepsilon.
\]
Taking positive parts gives
\[
(a_{k+1}-\varepsilon)_+
\le
(1-\theta_k)(a_k-\varepsilon)_+.
\]
Therefore, for every \(n\ge K_\varepsilon\),
\[
(a_{n+1}-\varepsilon)_+
\le
(a_{K_\varepsilon}-\varepsilon)_+
\prod_{k=K_\varepsilon}^{n}(1-\theta_k).
\]
Since
\[
\prod_{k=K_\varepsilon}^{n}(1-\theta_k)
\le
\exp\left(
-\sum_{k=K_\varepsilon}^{n}\theta_k
\right)
\longrightarrow0,
\]
we obtain
\[
\limsup_{k\to\infty}a_k\le\varepsilon.
\]
As \(\varepsilon>0\) is arbitrary and \(a_k\ge0\), it follows that
\[
a_k\longrightarrow0.
\]
Hence,
\[
\pmb{\rho^k}
\longrightarrow
\pmb p
=
\Pi_{\mathcal S}(\pmb{\rho^0}).
\]
Hence every subsequence converges to the same point \(\pmb p\), and
\(\pmb p\) is the unique accumulation point.

\end{proof}

\section{Proof of Proposition \ref{prop:a}}\label{app:prop}

In this section we will provide a proof of Proposition \ref{prop:a}. Proposition \ref{prop:a} follows directly from the following proposition.

\begin{proposition}[A matching $\Omega(T^{-1/2})$ lower-bound example]
\label{prop:lower-bound-example}
Consider the finite-horizon mean-field game with horizon $H=2$, state space
\[
S=\{0,1,2\},
\]
action space
\[
A=\{L,R\},
\]
and initial state distribution
\[
\mu_0=\delta_0.
\]
Let the transition kernel be independent of the state-measure and defined by
\[
P_0(s'\mid s,a)
=
\begin{cases}
1,
& s=0,\ a=L,\ s'=1,\\
1,
& s=0,\ a=R,\ s'=2,\\
1,
& s\in\{1,2\},\ s'=s,\\
0,
& \text{otherwise}.
\end{cases}
\]
Let the rewards be
\[
r(0,\cdot,\mu)\equiv 0,
\]
and
\[
r_1(1,\cdot,\mu)\equiv-\mu(1),\qquad r_1(2,\cdot,\mu)\equiv-\mu(2).
\]
Then, the following hold:
\begin{enumerate}
\item The game satisfies Assumptions \eqref{weakmon}--\eqref{lip}.
\item Let $\mathcal K$ and $F$ be the occupation-flow polytope and the mean-field occupation operator
associated with this MFG. Let $\pmb{\rho^0}\in \mathcal K$ be the occupation flow defined by
\begin{align*}
&\rho_0^0(0,L)=1,\qquad \rho_0^0(0,R)=0,\qquad
\rho_1^0(1,L)=\rho_1^0(1,R)=\frac12,
\\&\rho_1^0(2,L)=\rho_1^0(2,R)=0,
\end{align*}
with all other coordinates equal to $0$.
\item Consider the anchored iteration
\[
\pmb{\rho^{k+1}}
=
\Pi_{\mathcal K}\Bigl((1-\beta_k)\pmb{\rho^k}+\beta_k\pmb{\rho^0}-\alpha_kF(\pmb{\rho^k})\Bigr),
\]
with
\[
\alpha_k=\frac{1}{6\sqrt{k+2}},\qquad \beta_k=\frac{2}{k+2}.
\]
Then, for every $k\ge 1$,
\[
r^{\tan}_{F,N_{\mathcal K}}(\pmb{\rho^k})\ge \frac{1}{\sqrt{3(k+2)}}.
\]
In particular,
\[
r^{\tan}_{F,N_{\mathcal K}}(\pmb{\rho^k})=\Omega(k^{-1/2}).
\]
\end{enumerate}
\end{proposition}

\begin{proof}
We divide the proof into Steps.

\medskip

\noindent
By the definition of the mean-field occupation operator,
\[
F_h(\pmb{\rho})(s,a)
=
-r\bigl(s,a,\mu_h^{\pmb{\rho}}\bigr),
\qquad
h\in\{0,1\}.
\]

We first verify Assumptions \ref{weakmon}--\ref{lip}. Let
$\pmb{\rho},\pmb{\eta}\in\mathcal K$. Since
\[
r(0,a,\mu)=0,
\qquad
r(s,a,\mu)=-\mu(s),
\quad
s\in\{1,2\},
\]
we have
\begin{align*}
&\sum_{h=0}^1\sum_{s\in S}\sum_{a\in A}
\left(
r\bigl(s,a,\mu_h^{\pmb{\rho}}\bigr)
-
r\bigl(s,a,\mu_h^{\pmb{\eta}}\bigr)
\right)
\left(
\rho_h(s,a)-\eta_h(s,a)
\right)
\\
&=
-\sum_{h=0}^1\sum_{s=1}^2
\left(
\mu_h^{\pmb{\rho}}(s)-\mu_h^{\pmb{\eta}}(s)
\right)
\sum_{a\in A}
\left(
\rho_h(s,a)-\eta_h(s,a)
\right)
\\
&=
-\sum_{h=0}^1\sum_{s=1}^2
\left(
\mu_h^{\pmb{\rho}}(s)-\mu_h^{\pmb{\eta}}(s)
\right)^2
\le 0.
\end{align*}
Thus, the reward function is weakly monotone. Notice that the
time-$0$ contribution is zero because
\[
\mu_0^{\pmb{\rho}}
=
\mu_0^{\pmb{\eta}}
=
\delta_0.
\]

Moreover, for every $s\in\{1,2\}$ and $a\in A$,
\[
\left|
r(s,a,\mu)-r(s,a,\nu)
\right|
=
|\mu(s)-\nu(s)|
\le
\|\mu-\nu\|_2,
\]
whereas
\[
\left|
r(0,a,\mu)-r(0,a,\nu)
\right|
=
0.
\]
Hence, Assumption \ref{lip} holds with Lipschitz constant $1$.

We next study the anchored iterations. After removing coordinates that
are identically zero for every element of $\mathcal K$, we use the
coordinate ordering
\[
\left(
\rho_0(0,L),
\rho_0(0,R),
\rho_1(1,L),
\rho_1(1,R),
\rho_1(2,L),
\rho_1(2,R)
\right).
\]
In these coordinates,
\[
\mathcal K
=
\left\{
(x,1-x,u,v,w,z):
\begin{array}{l}
0\le x\le 1,\\
u,v,w,z\ge 0,\\
u+v=x,\\
w+z=1-x
\end{array}
\right\}.
\]

For $p\in[0,1]$, define
\[
\pmb{\rho}(p)
:=
\left(
p,1-p,\frac p2,\frac p2,
\frac{1-p}{2},\frac{1-p}{2}
\right).
\]
The occupation measure specified in the statement satisfies
\[
\pmb{\rho^0}=\pmb{\rho}(1).
\]

The state distribution induced at time $1$ by $\pmb{\rho}(p)$ is
\[
\mu_1^{\pmb{\rho}(p)}(1)=p,
\qquad
\mu_1^{\pmb{\rho}(p)}(2)=1-p.
\]
Since the reward function is time homogeneous and
$\mu_0=\delta_0$, we have
\[
F_0\bigl(\pmb{\rho}(p)\bigr)(0,L)
=
F_0\bigl(\pmb{\rho}(p)\bigr)(0,R)
=
0.
\]
At time $1$,
\[
F_1\bigl(\pmb{\rho}(p)\bigr)(1,a)=p,
\qquad
F_1\bigl(\pmb{\rho}(p)\bigr)(2,a)=1-p,
\qquad
a\in A.
\]
Consequently,
\[
F\bigl(\pmb{\rho}(p)\bigr)
=
(0,0,p,p,1-p,1-p).
\]

We now show that the iterates remain in the one-dimensional set
$\mathcal M:=\left\{\pmb{\rho}(p):p\in[0,1]\right\}.$
Suppose that $\pmb{\rho^k}=\pmb{\rho}(p_k)$ for some $p_k\in[0,1]$, and define
\[a_k:=(1-\beta_k)p_k+\beta_k.\]
Then
\begin{align*}
\pmb y^k
&:=
(1-\beta_k)\pmb{\rho^k}
+
\beta_k\pmb{\rho^0}
-
\alpha_kF(\pmb{\rho^k})
\\
&=
\left(
a_k,1-a_k,c_k,c_k,d_k,d_k
\right),
\end{align*}
where
\[
c_k
=
\frac{a_k}{2}-\alpha_kp_k
\]
and
\[
d_k
=
\frac{1-a_k}{2}-\alpha_k(1-p_k).
\]

For a fixed $x\in[0,1]$, the minimizers of
\[
(u-c_k)^2+(v-c_k)^2
\]
subject to $u+v=x$ are
\[
u=v=\frac{x}{2}.
\]
Similarly, the minimizers of
\[
(w-d_k)^2+(z-d_k)^2
\]
subject to $w+z=1-x$ are
\[
w=z=\frac{1-x}{2}.
\]
Therefore, the projection of $\pmb y^k$ onto $\mathcal K$ is of the
form $\pmb{\rho}(x)$, where $x$ minimizes
\[
\varphi_k(x)
=
2(x-a_k)^2
+
2\left(\frac{x}{2}-c_k\right)^2
+
2\left(\frac{1-x}{2}-d_k\right)^2
\]
over $x\in[0,1]$.

Differentiating gives
\[
\varphi_k'(x)
=
6x-6a_k+4\alpha_kp_k-2\alpha_k.
\]
Thus, the unconstrained minimizer is
\[
x
=
a_k-\frac{2}{3}\alpha_kp_k+\frac{\alpha_k}{3}.
\]
It follows that
\begin{equation}
\label{eq:lower-bound-p-recursion}
p_{k+1}
=
\left(
1-\beta_k-\frac{2}{3}\alpha_k
\right)p_k
+
\beta_k
+
\frac{\alpha_k}{3}.
\end{equation}

We next verify that the unconstrained minimizer belongs to $(0,1)$.
For $k=0$, since $\beta_0=1$ and $p_0=1$,
\[
p_1
=
1-\frac{\alpha_0}{3}
=
1-\frac{1}{18\sqrt{2}}
\in\left(\frac12,1\right).
\]
For every $k\ge1$,
\[
1-\beta_k-\frac{2}{3}\alpha_k
=
1-\frac{2}{k+2}-\frac{1}{9\sqrt{k+2}}
>0.
\]
Suppose that $p_k\in(1/2,1)$. Then
\begin{align*}
p_{k+1}
&>
\frac12
\left(
1-\beta_k-\frac{2}{3}\alpha_k
\right)
+
\beta_k+\frac{\alpha_k}{3}
\\
&=
\frac12+\frac{\beta_k}{2}
>
\frac12.
\end{align*}
Moreover,
\begin{align*}
p_{k+1}
&<
1-\beta_k-\frac{2}{3}\alpha_k
+
\beta_k+\frac{\alpha_k}{3}
\\
&=
1-\frac{\alpha_k}{3}
<1.
\end{align*}
Therefore,
\[
p_k\in\left(\frac12,1\right)
\qquad
\text{for every }k\ge1.
\]
In particular, the projection does not reach the boundary of the
interval, and \eqref{eq:lower-bound-p-recursion} holds for every
$k\ge0$.

Define
\[
e_k:=p_k-\frac12.
\]
Subtracting $1/2$ from both sides of
\eqref{eq:lower-bound-p-recursion} gives
\[
e_{k+1}
=
\left(
1-\beta_k-\frac{2}{3}\alpha_k
\right)e_k
+
\frac{\beta_k}{2}.
\]
Using
\[
\alpha_k=\frac{1}{6\sqrt{k+2}},
\qquad
\beta_k=\frac{2}{k+2},
\]
we obtain
\begin{equation}
\label{eq:lower-bound-e-recursion}
e_{k+1}
=
\left(
1-\frac{2}{k+2}
-\frac{1}{9\sqrt{k+2}}
\right)e_k
+
\frac{1}{k+2}.
\end{equation}

We claim that
\begin{equation}
\label{eq:lower-bound-e-estimate}
e_k
\ge
\frac{1}{2\sqrt{k+2}}
\qquad
\text{for every }k\ge1.
\end{equation}
Indeed, for $k=1$,
\[
e_1
=
\frac12-\frac{1}{18\sqrt2}
>
\frac{1}{2\sqrt3}.
\]
Suppose that \eqref{eq:lower-bound-e-estimate} holds for some $k\ge1$,
and set
\[
n:=k+2.
\]
Since $n\ge3$,
\[
1-\frac{2}{n}-\frac{1}{9\sqrt n}>0.
\]
Hence, by \eqref{eq:lower-bound-e-recursion},
\begin{align*}
e_{k+1}
&\ge
\left(
1-\frac{2}{n}-\frac{1}{9\sqrt n}
\right)
\frac{1}{2\sqrt n}
+
\frac1n
\\
&=
\frac{1}{2\sqrt n}
-
\frac{1}{n\sqrt n}
+
\frac{17}{18n}.
\end{align*}
Furthermore,
\begin{align*}
&
\frac{1}{2\sqrt n}
-
\frac{1}{n\sqrt n}
+
\frac{17}{18n}
-
\frac{1}{2\sqrt{n+1}}
\\
&=
\left(
\frac{1}{2\sqrt n}
-
\frac{1}{2\sqrt{n+1}}
\right)
+
\frac1n
\left(
\frac{17}{18}-\frac{1}{\sqrt n}
\right)
\ge0.
\end{align*}
Therefore,
\[
e_{k+1}
\ge
\frac{1}{2\sqrt{n+1}}
=
\frac{1}{2\sqrt{k+3}}.
\]
This proves \eqref{eq:lower-bound-e-estimate} by induction.

It remains to compute the tangent residual. Fix $p\in(0,1)$. Since
all coordinates of $\pmb{\rho}(p)$ that are not identically zero on
$\mathcal K$ are strictly positive,
\[
\pmb{\rho}(p)\in\operatorname{ri}(\mathcal K),
\]
where $\mathrm{ri}(\mathcal K)$ denotes the relative interior.
The tangent space to $\mathcal K$ at $\pmb{\rho}(p)$ is
\[
T_{\mathcal K}
=
\left\{
(\xi,-\xi,\zeta_1,\zeta_2,\zeta_3,\zeta_4):
\zeta_1+\zeta_2=\xi,\ 
\zeta_3+\zeta_4=-\xi
\right\}.
\]
An orthogonal basis of $T_{\mathcal K}$ is
\[
\pmb t
=
\left(
1,-1,\frac12,\frac12,-\frac12,-\frac12
\right),
\]
together with
\[
\pmb d_1=(0,0,1,-1,0,0) \text{ and } \pmb d_2=(0,0,0,0,1,-1).
\]
Since
\[
F\bigl(\pmb{\rho}(p)\bigr)
=
(0,0,p,p,1-p,1-p),
\]
we have
\[
\left\langle
F\bigl(\pmb{\rho}(p)\bigr),\pmb d_1
\right\rangle
=
\left\langle
F\bigl(\pmb{\rho}(p)\bigr),\pmb d_2
\right\rangle
=
0.
\]
Moreover,
\[
\left\langle
F\bigl(\pmb{\rho}(p)\bigr),\pmb t
\right\rangle
=
2p-1
\]
and
\[
\|\pmb t\|_2^2=3.
\]
Consequently,
\[
\left\|
\Pi_{T_{\mathcal K}}
F\bigl(\pmb{\rho}(p)\bigr)
\right\|_2
=
\frac{|2p-1|}{\sqrt3}.
\]

Since
\[
N_{\mathcal K}\bigl(\pmb{\rho}(p)\bigr)
=
T_{\mathcal K}^{\perp},
\]
the tangent residual satisfies
\begin{align*}
r^{\tan}_{F,N_{\mathcal K}}
\bigl(\pmb{\rho}(p)\bigr)
&=
\operatorname{dist}
\left(
\pmb 0,
F\bigl(\pmb{\rho}(p)\bigr)
+
N_{\mathcal K}\bigl(\pmb{\rho}(p)\bigr)
\right)
\\
&=
\left\|
\Pi_{T_{\mathcal K}}
F\bigl(\pmb{\rho}(p)\bigr)
\right\|_2
\\
&=
\frac{|2p-1|}{\sqrt3}.
\end{align*}

Finally, for every $k\ge1$,
\[
p_k=\frac12+e_k
\]
with $e_k>0$. Therefore,
\begin{align*}
r^{\tan}_{F,N_{\mathcal K}}(\pmb{\rho^k})
&=
\frac{|2p_k-1|}{\sqrt3}
\\
&=
\frac{2e_k}{\sqrt3}
\\
&\ge
\frac{1}{\sqrt{3(k+2)}}.
\end{align*}
Hence,
\[
r^{\tan}_{F,N_{\mathcal K}}(\pmb{\rho^k})
=
\Omega(k^{-1/2}),
\]
as claimed.
\end{proof}
\begin{figure}[h]
    \centering
    \begin{minipage}{0.48\linewidth}
        \centering
        \includegraphics[width=\linewidth]{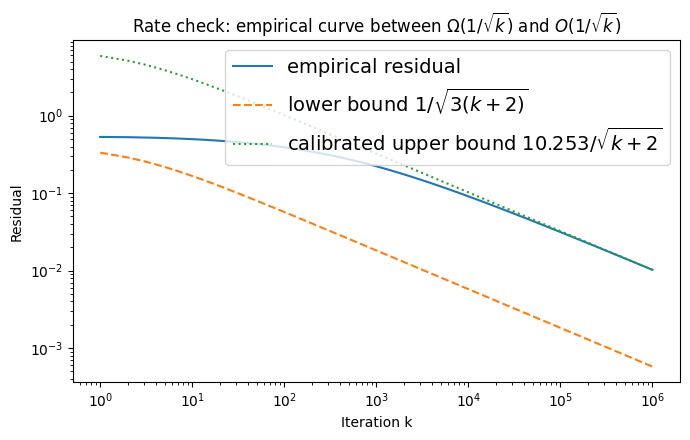}
        \caption{Empirical residual error rate for the example we use for Proposition \ref{prop:a}.}
        \label{fig:first}
    \end{minipage}
    \hfill
    \begin{minipage}{0.48\linewidth}
        \centering
        \includegraphics[width=\linewidth]{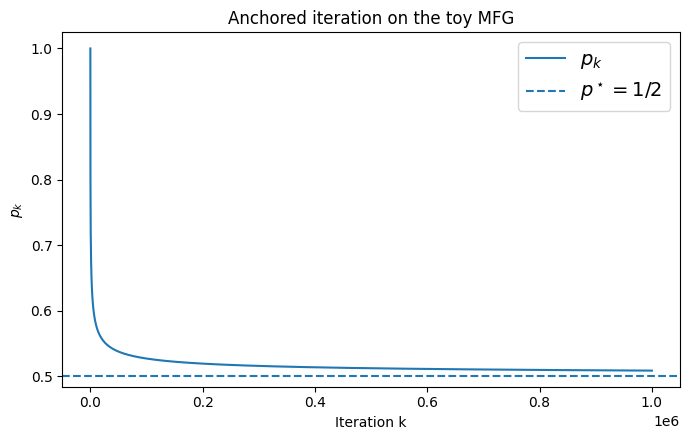}
        \caption{Comparison between the iterates produced by our algorithm for the example in Proposition \ref{prop:a} and the corresponding MFE.}
        \label{fig:second}
    \end{minipage}
\end{figure}

\end{document}